\documentclass[12pt, reqno]{amsart}
\usepackage[a4paper,top=2cm,bottom=2cm,left=2cm,right=2cm,marginparwidth=2cm]{geometry}
\usepackage[utf8]{inputenc}
\usepackage[english]{babel}
\usepackage[T1]{fontenc}
\usepackage{graphicx,caption,subcaption}
\usepackage{amsmath,amsthm, amssymb}
\usepackage{mathabx,epsfig}
\usepackage{todonotes}
\usepackage{comment}
\usepackage[colorlinks=true, allcolors=blue]{hyperref}
\usepackage{enumitem}

\usepackage{tikz-cd}
\usepackage{tikz}
\usetikzlibrary{matrix}

\def\from{\colon}

\def\Julia{\mathcal{J}}

\def\Rat{\mathrm{Rat}}
\def\QB{\mathcal{QB}}

\def\Cbb{\mathbb{C}}

\def\Hbb{\mathbb{H}}
\def\Rbb{\mathbb{R}}

\def\Bcal{\mathcal{B}}
\def\Ecal{\mathcal{E}}
\def\Hcal{\mathcal{H}}
\def\Lcal{\mathcal{L}}
\def\Pcal{\mathcal{P}}

\DeclareMathOperator{\SL}{SL}
\DeclareMathOperator{\PSL}{PSL}
\DeclareMathOperator{\Aut}{Aut}
\DeclareMathOperator{\QF}{\mathcal{QF}}

\newtheorem{thm}{Theorem}[section]
\newtheorem{coro}[thm]{Corollary}
\newtheorem{lem}[thm]{Lemma}
\newtheorem{prop}[thm]{Proposition}

\theoremstyle{definition}
\newtheorem{defn}[thm]{Definition}

\theoremstyle{remark}

\title{Pressure metrics in Geometry and Dynamics}
\author{Yan Mary He}
\address{Department of Mathematics, University of Oklahoma, Norman, OK 73019, USA}
\email{he@ou.edu}
\author{Homin Lee}
\address{Department of Mathematics, Northwestern University, Evaston, IL 60208, USA}
\email{homin.lee@northwestern.edu}
\author{Insung Park}
\address{Institute for Mathematical Sciences, Stony Brook University, Stony Brook NY 11794-3660, USA}
\email{insung.park@stonybrook.edu}
\begin{document}

\begin{abstract}
In this article, we first survey  results on pressure metrics on various deformation spaces in geometry, topology, and dynamics. We then discuss pressure semi-norms and their degeneracy loci in the space of quasi-Blaschke products. 
\end{abstract}

\maketitle
\tableofcontents

\section{Introduction}
Suppose that $S$ is a closed orientable surface of genus $g\ge2$. The Teichm\"uller space $T(S)$ of $S$, which is the space of holomorphic 
structures on $S$, plays a fundamental role in modern mathematics. The topology and geometry of the Teichm\"uller space have been investigated from numerous viewpoints. Ahlfors proved that the Teichm\"uller space $T(S)$ is homeomorphic to $\mathbb R^{6g-6}$ \cite[Theorem 14]{Ahlfors_QCMappings}. 
Bers proved that $T(S)$ is biholomorphic to an open bounded domain in $\mathbb{C}^{3g-3}$ \cite{Bers_EmbBdd}. 
There are a number of metrics on the Teichm\"uller space $T(S)$ defined from different perspectives, e.g., the Teichm\"uller metric, the Weil-Petersson metric, and the Thurston metric. We refer the interested readers to \cite{ImaTan,Hub_vol1} for comprehensive accounts of Teichm\"uller spaces.

Wolpert proved that the Weil-Petersson metric on $T(S)$ can also be obtained {\color{black} by means of} the second-order derivatives of the length functions for generic geodesics \cite{Wolpert_WPmetricbyThurston}. More recently, Bridgeman and McMullen showed that the Weil-Petersson metric can be reconstructed using thermodynamic formalism \cite{Bridgeman_WPMetricQF,McMullen08}. More precisely, they proved that the Weil-Petersson metric is a constant multiple of the so-called {\it pressure metric}.

Since then, the idea of constructing pressure metrics has been applied to various deformation spaces in geometry and topology. For example, pressure metrics have been {\color{black} constructed and }
studied for quasi-Fuchsian spaces of closed orientable surfaces \cite{BridgemanTaylor2, Bridgeman_WPMetricQF}, Teichm\"uller spaces and quasi-Fuchsian spaces of punctured surfaces \cite{Kao_PressureMetTeichPuncSurf, BCK_PressureMetricQFPunctred}, Teichm\"uller spaces of bordered surfaces \cite{Xu_PressureMetricBordered}, and deformation spaces of Anosov representations \cite{BCLS}. Pressure metrics have also been defined on the moduli space of metric graphs \cite{PollicottSharp_WPMetricMetricGraph} and on the Culler-Vogtmann outer spaces \cite{Aougab23}. In Sections \ref{sec:McMullen}--\ref{sec:AnosovReps}, we will discuss the results of McMullen \cite{McMullen08}, Bridgeman-Taylor \cite{BridgemanTaylor2}, Bridgeman \cite{Bridgeman_WPMetricQF}, and Bridgeman-Canary-Labourie-Sambarino \cite{BCLS}.

According to Sullivan's dictionary, Blaschke products in complex dynamics can be considered as an analogue of hyperbolic surfaces. The space of degree-$d$ Blaschke products $\Bcal_d$ corresponds to the Teichm\"uller space $T(S)$. The analogies between $\Bcal_d$ and $T(S)$ have been studied in \cite{McM_DynDisk}, and the degeneration of Blaschke products and the boundary of $\Bcal_d$ have been investigated in \cite{McM_CompExpCircle,McM_RibbonRtreeHoloDyn,Luo_GeoFiniteDegenI}.

However, compared to metrics on Teichm\"uller spaces, metrics on the space of Blaschke products are less {\color{black} studied}. McMullen introduced the first metric on $\Bcal_d$ in \cite{McMullen08} using thermodynamic formalism, analogous to the construction of the pressure metric on the Teichm\"uller space $T(S)$. Ivrii studied the completion of this metric for the degree-$2$ case in \cite{Ivrii} in analogy to augmented Teichm\"uller spaces \cite{Bers_AugTeich}, which is the completion of Teichm\"uller spaces with respect to Weil-Petersson metrics \cite{Masur_WPCompletion}.
Nie and the first author constructed pressure metrics on the hyperbolic components in the moduli space of degree-$d$ rational maps for $d\ge2$ \cite{HeNie_MetricHypComp}. Section \ref{sec:HeNie} will summarize the results from \cite{HeNie_MetricHypComp}. 


In general, using thermodynamic formalism, we obtain positive semi-definite symmetric bilinear 2-forms $\langle\cdot,\cdot\rangle_\Pcal$ and semi-norms $||\cdot||_\Pcal$ on deformation spaces. In some cases, the 2-forms $\langle\cdot,\cdot\rangle_\Pcal$ are positive definite; that is, they are Riemannian metrics. In general, however, these forms may have degenerate vectors. Hence, we refer to $\langle\cdot,\cdot\rangle_\Pcal$ as {\it pressure forms} and $||\cdot||_\Pcal$ {\it pressure semi-norms}.

The degenerating vectors of the pressure semi-norms on the spaces of quasi-Fuchsian groups are characterized in \cite{Bridgeman_WPMetricQF}. However, the degenerating vectors of the pressure semi-norms on the spaces of Blaschke and quasi-Blaschke products remain unknown. In Section \ref{sec_DegeLoci_QB}, we investigate the degeneracy loci of pressure semi-norms on deformation spaces of quasi-Blaschke products.

\subsection*{Acknowledgements} The authors would like to thank Curt McMullen and Oleg Ivrii for useful conversations. H.\ L.\ was supported by an AMS-Simons Travel Grant. The authors would also like to thank the referee for a careful reading of the manuscript and for many useful comments.

\section{Pressure metrics in geometry and dynamics}\label{sec:Survey}
In this section, we survey results on pressure metrics in various deformation spaces of geometric structures and holomorphic dynamical systems in \cite{Bowen74,BCLS,HeNie_MetricHypComp,McMullen08}.

\subsection{Thermodynamic formalism}\label{sec:ThermodynFormalism}
In this subsection, we give a brief introduction to the pressure metric in the thermodynamic setting. Standard references are \cite{McMullen08, Parry90, Ruelle04}.

Fix an integer $n \ge 1$ and an $n\times n$ aperiodic matrix $A$ with entries equal to either $1$ or $0$. Recall that a matrix $A$ is {\it aperiodic} if there exists $k \in \mathbb N$ such that every entry of $A^k$ is positive. We define the {\it one-sided subshift of finite type} $(\Sigma_A^+,\sigma)$ as follows. We first define a set $\Sigma_A^+$ by
$$\Sigma_A^+ := \{ \underline{i} = (i_0, i_1, \ldots) ~|~  i_j \in \{1, \ldots, n \}, A_{i_j,i_{j+1}} = 1\}.$$
There is a standard metric $d_\Sigma$ on $\Sigma_A^+$ defined as 
$$d_\Sigma(\underline{x},\underline{\smash{y}}) := 2^{-N(\underline{x},\underline{\smash{y}})},$$ where $N(\underline{x},\underline{\smash{y}}) := \min\{j~|~ x_j \neq y_j \}$. 
With respect to {\color{black} the topology induced by} this metric, $\Sigma_A^+$ is a compact metric space.
We define the {\it shift map} $\sigma \from \Sigma_A^+ \to \Sigma_A^+$ as $$\sigma( i_0, i_1, i_2, \ldots) := (i_1, i_2, i_3, \ldots).$$ 

For $\alpha \in (0,1]$, a continuous function $\phi \from \Sigma_A^+ \to \mathbb{R}$ is {\it $\alpha$-H\"older continuous} if there exists a constant $C>0$ such that $$|\phi(\underline{x}) - \phi(\underline{\smash{y}})| \le C  d(\underline{x},\underline{\smash{y}})^{\alpha}$$ for any $\underline{x},\underline{\smash{y}} \in \Sigma_A^+$.
Denote by $C^{\alpha}(\Sigma_A^+)$ the space of $\alpha$-H\"older continuous real-valued functions on $\Sigma_A^+$.  We say that a continuous function $\phi \from \Sigma_A^+ \to \mathbb R$ is {\it H\"older continuous} if it is $\alpha$-H\"older continuous for some $\alpha \in (0,1]$. 

For $\phi \in C^{\alpha}(\Sigma_A^+)$, the {\it transfer operator} $\mathcal{L}_{\phi} \from C^{\alpha}(\Sigma_A^+) \to C^{\alpha}(\Sigma_A^+)$ is defined by 
$$\mathcal{L}_{\phi}(g)(y) := \sum_{\sigma(x) = y} e^{\phi(x)}  g(x).$$
\textcolor{black}{Even though the number of preimages of $\sigma$ may not be constant, $\mathcal{L}_{\phi}(g)$ is $\alpha$-H\"older continuous when $g$ is $\alpha$-H\"older continuous.}
By the Ruelle-Perron-Frobenius theorem, there is a positive eigenfunction $e^{\psi}$, unique up to scale, such that $$\mathcal{L}_{\phi}(e^{\psi}) = \rho(\mathcal{L}_{\phi}) e^{\psi}.$$
The spectral radius $\rho(\mathcal{L}_{\phi})$ is an isolated eigenvalue so that the rest of the spectrum is contained in a disk of radius $r < \rho(\mathcal{L}_{\phi})$.

The {\it pressure} $\mathcal{P}(\phi)$ of $\phi$ is defined by $$\mathcal{P}(\phi) := \log \rho(\mathcal{L}_{\phi}).$$
Alternatively, the pressure $\mathcal{P}(\phi)$ can also be defined using variational methods. Denote by $\mathcal{M}_\sigma$ the set of $\sigma$-invariant probability measures on $\Sigma_A^+$. Then, {\color{black} we have}
$$\mathcal{P}(\phi)  = \sup_{m \in \mathcal{M}_\sigma} \left(h_m(\sigma) + \int_{\Sigma_A^+} \phi \, dm \right),$$
where $h_m(\sigma)$ is the measure-theoretic entropy of $\sigma$ with respect to $m\in \mathcal{M}_\sigma$. A measure $m=m(\phi) \in \mathcal{M}_\sigma$ is called an {\it equilibrium state} or {\it equilibrium measure} of $\phi$ if
\[
    \mathcal{P}(\phi)  =h_m(\sigma) + \int_{\Sigma_A^+} \phi \, dm.
\]
It is well-known that every $\phi \in C^{\alpha}(\Sigma_A^+)$ has a unique equilibrium measure. Note that the equilibrium measure $m(\phi)$ is an ergodic $\sigma$-invariant probability measure with positive entropy; see \cite{ClimCyr}.

The equilibrium measure $m(\phi)$ is also related to the spectral data of transfer operators {\color{black} described above}. If $\mathcal{P}(\phi) = 0$, then $\mathcal{L}_{\phi}(e^{\psi}) = e^{\psi}$. There is a positive measure $\mu = \mu(\phi)$ on $\Sigma_A^+$ that is uniquely determined as the eigenmeasure of the dual linear operator $\Lcal_\phi^* \from \mathcal{M}_\sigma \to \mathcal{M}_\sigma$ with eigenvalue one, i.e.,
$$\int_{\Sigma_A^+} \mathcal{L}_{\phi}(\tilde\phi) \, d\mu = \int_{\Sigma_A^+}\tilde\phi \, d\mu \hspace{20pt} \text{for~all~} \tilde\phi\in C^{\alpha}(\Sigma_A^+)$$
and $\int_{\Sigma_A^+} e^{\psi} \, d\mu =1$. We have
$$m(\phi) = e^{\psi} \mu(\phi).$$

The {\it asymptotic variance} (which is called variance in \cite{McMullen08}) of a H\"older continuous function $\psi \from \Sigma_A^+ \to \mathbb R$ is given by 
$${\rm Var}(\psi, m(\phi)) = \lim_{n \to \infty} \frac{1}{n} \int_{\Sigma_A^+} \left( \sum_{i=0}^{n-1} \psi \circ \sigma^i(x) \right)^2 \, dm(\phi).$$

In what follows, we denote by $\dot{\phi_0}$ and $\ddot{\phi_0}$ the functions $\dot{\phi_0}\from \Sigma_A^+ \to \mathbb R$ and $\ddot{\phi_0}\from \Sigma_A^+ \to \mathbb R$ defined by $\dot{\phi_0}(z) := \left.\frac{d}{dt}\right|_{t=0}\phi_t(z)$ and $\ddot{\phi_0}(z) := \left.\frac{d^2}{dt^2}\right|_{t=0}\phi_t(z)$.

By using \cite[Propositions 4.10 and 4.11]{Parry90}, we obtain the following identities.

\begin{prop}[{\cite[Theorem 2.2]{McMullen08}}] \label{prop_thm2.2}
Fix a smooth path $\phi_t$ in $C^{\alpha}(\Sigma_A^+)$. Denote the equilibrium measure of $\phi_0$ by $m = m(\phi_0)$. Then we have
$$\frac{d\mathcal{P}(\phi_t)}{dt} \bigg|_{t=0} = \int_{\Sigma_A^+} \dot{\phi_0} \,dm,$$
and, if the first derivative of $\mathcal{P}(\phi_t)$ {\color{black} at $t=0$} is zero, then
$$\frac{d^2\mathcal{P}(\phi_t)}{dt^2} \bigg|_{t=0} = {\rm Var}(\dot{\phi_0}, m)  + \int_{\Sigma_A^+} \ddot{\phi_0} \,dm.$$
\end{prop}

Two continuous functions $\phi_1$ and $\phi_2$ are {\it cohomologous}, denoted by $\phi_1 \sim \phi_2$, if there exists a continuous function $h \from \Sigma_A^+ \to \mathbb R$ such that $\phi_1(x) - \phi_2(x) = h(\sigma(x)) - h(x)$. 
The pressure function $\mathcal{P} \from C^{\alpha}(\Sigma_A^+) \to \mathbb R$ depends only on the cohomology class of $\phi$. Denote by $\mathcal{C}(\Sigma_A^+)$ the set of cohomology classes of H\"older continuous functions with zero pressure, that is,
$$\mathcal{C}(\Sigma_A^+) := \left\{\phi : \phi \in C^{\alpha}(\Sigma_A^+) \text{ for some } \alpha \in (0,1], \mathcal{P}(\phi) = 0 \right \}/ \sim.$$

For $[\phi] \in\mathcal{C}(\Sigma_A^+)$, denote by $m$ the equilibrium measure for a representative (hence, any representative) $\phi$ in $[\phi]$. The tangent space of $\mathcal{C}(\Sigma_A^+)$ at $[\phi]$ can be identified with 
$$T_{[\phi]}\mathcal{C}(\Sigma_A^+) = \left\{\psi ~\bigg|~\psi\ \text{is~H\"older continuous}, \int_{\Sigma_A^+} \psi \, dm = 0  \right\} / \sim.$$
{\color{black} The quotient on the right hand side of the above formula is well-defined as the pressure function depends only on the cohomology classes of H\"older continuous functions.}

For a $[\psi] \in T_{[\phi]}\mathcal{C}(\Sigma_A^+)$, we define the {\it pressure norm} $||\cdot||_{pm}$ on $T_{[\phi]}\mathcal{C}(\Sigma_A^+)$ by
\begin{equation}\label{eq_def_pm}
    ||[\psi]||^2_{pm}:= \frac{{\rm Var}(\psi,m)}{-\int_{\Sigma_A^+} \phi \, dm}
\end{equation}
following \cite[p.375]{McMullen08}.
Note that the denominator of Equation \eqref{eq_def_pm} is positive as $0 = \mathcal{P}(\phi) = \int_{\Sigma_A^+} \phi \,dm + h_{m}(\sigma)$ and $h_{m}(\sigma)>0$.

The non-degeneracy of $||\cdot||_{pm}$ follows from the strict convexity of $\mathcal{P}$: The second derivative
$$\frac{d^2\mathcal{P}(\phi + t \psi)}{dt^2} \bigg|_{t=0} = {\rm Var}(\psi, m(\phi))$$ is non-negative {\color{black} for any $[\phi] \in \mathcal{C}(\Sigma_A^+)$ and any $[\psi] \in T_{[\phi]}\mathcal{C}(\Sigma_A^+)$}, and is zero if and only if $\psi$ is cohomologous to zero; see \cite[Proposition 4.12]{Parry90}. 


Moreover, we can define the {\it pressure form}  $\langle\cdot,\cdot\rangle_{pm}$ on $\mathcal{C}(\Sigma_A^+)$ by
\[
    \langle[\psi_1],[\psi_2]\rangle_{pm}:= \frac{{\rm Cov}(\psi_1,\psi_2,m)}{-\int_{\Sigma_A^+} \phi \, dm},
\]
where $[\psi_1],[\psi_2]\in T_{[\phi]}\mathcal{C}(\Sigma_A^+)$ and
$${\rm Cov}(\psi_1,\psi_2, m(\phi)) = \lim_{n \to \infty} \frac{1}{n} \int_{\Sigma_A^+} \left(\sum_{i=0}^{n-1} \psi_1 \circ \sigma^i(x)\right)\left(\sum_{i=0}^{n-1} \psi_2 \circ \sigma^i(x)\right)  \, dm(\phi).$$
Then we have $\langle[\psi],[\psi]\rangle_{pm}=||[\psi]||_{pm}^2$. 



As we will see in Section \ref{sec:McMullen}, the pressure form induced in the Teichm\"uller space is a constant multiple of the Weil-Petersson (Riemannian) metric. This provides a motivation to study pressure forms and their geometric properties in other moduli spaces of geometric structures or dynamical systems.

\subsection*{ Terminologies: pressure forms, semi-norms, and metrics} In what follows, for deformation spaces $\mathcal{D}$ in diverse contexts, we construct maps $\mathcal{E}\from \mathcal{D}\to \mathcal{C}(\Sigma_A^+)$. We define $\langle\cdot,\cdot\rangle_\Pcal:=\mathcal{E}^*\langle\cdot,\cdot\rangle_{pm}$ and $||\cdot ||_\Pcal:= \mathcal{E}^*||\cdot||_{pm}$, where $\mathcal{E}^*$ means the pullback via $\mathcal{E}$. If $\mathcal{E}$ is not an immersion, then the non-degeneracy is not preserved by $\mathcal{E}^*$. That is, in general $\langle\cdot,\cdot\rangle_\Pcal$ is a positive semi-definite symmetric bilinear 2-form and $||\cdot ||_\Pcal$ is a semi-norm. Hence, we will call $\langle\cdot,\cdot\rangle_\Pcal$ and $||\cdot ||_\Pcal$ the {\it pressure form} and the {\it pressure semi-norm}, respectively. We refer to the path pseudo-metric $d_\Pcal$ defined by
\[
    d_\Pcal(x,y):=\inf_\gamma \int_\gamma ||\gamma'||_\Pcal,
\]
where the infimum is taken over $C^1$-curves between $x$ and $y$, as the {\it pressure pseudo-metric} on $\mathcal{D}$.  It is worth noting that, even when $||\cdot ||_\Pcal$ has degenerating vectors, $d_\Pcal$ may be a metric, see \cite[Lemma 13.1]{BCLS}. By abuse of notation, when $\langle\cdot,\cdot\rangle_\Pcal$ is positive-definite, or, equivalently, $||\cdot ||_\Pcal$ is a norm, we also call them pressure metrics.




\subsection{Pressure metrics on Teichm\"uller spaces and the space of Blaschke products}\label{sec:McMullen}
In \cite{McMullen08}, McMullen showed that the Weil-Petersson metric on the Teichm\"uller space $T(S)$ of a closed orientable genus $g \ge 2$ surface $S$ is equivalent to various quantities involving Hausdorff dimensions of dynamical {\color{black} objects}. From the perspective of Sullivan's dictionary, McMullen also established parallel results for the space $\mathcal{B}_d$ of degree-$d$ Blaschke products for $d\ge 2$ in complex dynamics. {\color{black} Recall that }{\it Blaschke products} are degree-$d$ proper holomorphic self-maps of the unit disk $\Delta$, and $\mathcal{B}_d$ is {\color{black} the quotient space of} the set of degree-$d$ Blaschke products by the conjugate action of $\Aut_{\mathbb{C}}(\Delta)$. 
In this subsection, we survey the results of \cite{McMullen08}.
{\color{black} In particular, in Section \ref{sec_Curt_1}, we discuss the results on Teichm\"uller spaces, and in Section \ref{sec_Curt_2}, we discuss parallel results for Blaschke products.}

\subsubsection{Pressure metrics on Teichm\"uller spaces} \label{sec_Curt_1}
\subsection*{Teichm\"uller spaces and Weil-Petersson metrics}
Fix a closed orientable surface $S$ of genus $g \ge 2$. Consider a diffeomorphism $\varphi \from S \to Y$ where $Y$ is a surface with a complete hyperbolic metric. We call the pair $(Y,\varphi)$ a {\it marked hyperbolic surface} and the diffeomorphism $\varphi$ a {\it marking}. Two marked hyperbolic structures $\varphi_1 \from S \to Y_1$ and $\varphi_2 \from S \to Y_2$ are {\it homotopic} if there is an isometry $I\from Y_1 \to Y_2$ such that $I \circ \varphi_1 \from S \to Y_2$ and $\varphi_2 \from S \to Y_2$ are homotopic. The {\it Teichm\"uller space} $T(S)$ of $S$ is defined as the set of homotopy classes of marked hyperbolic surfaces.

There is an alternative definition of the Teichm\"uller space $T(S)$. Denote by $\pi_1 (S)$ the fundamental group of $S$. The orientation-preserving isometry group of the hyperbolic $2$-space $\mathbb H^2$ can be identified with the Lie group $\PSL(2,\mathbb R)$. {\color{black} Recall that} a representation (i.e., a group homomorphism) $\rho \from \pi_1 (S) \to \PSL(2, \mathbb R)$ is {\it discrete} if the image $\rho(\pi_1(S))$ is discrete in $\PSL(2, \mathbb R)$ and is {\it faithful} if it is injective. The group ${\rm PGL}(2,\mathbb R)$ acts on the space ${\rm DF}(\pi_1(S), \PSL(2, \mathbb R))$ of discrete faithful representations $\rho \from \pi_1 (S) \to \PSL(2, \mathbb R)$ by conjugation, i.e., for $h \in {\rm PGL}(2,\mathbb R), \gamma \in \pi_1(S)$, {\color{black} we have}
$$(h \cdot \rho)(\gamma) = h\cdot \rho(\gamma) \cdot h^{-1}.$$
The Teichm\"uller space $T(S)$ can be identified with the quotient
$$T(S) \cong {\rm DF}(\pi_1(S), \PSL(2, \mathbb R))\,/\,{\rm PGL}(2,\mathbb R).$$
A {\it Fuchsian group} is a discrete subgroup of $\PSL(2,\mathbb R)$. \textcolor{black}{Hence, for $[\rho] \in T(S)$, its image $\rho(\pi_1(S))$ is a Fuchsian group.}

{\color{black} Recall that $T(S)$ is a smooth manifold \cite{Bers_EmbBdd}.} \textcolor{black}{Fix} $X \in T(S)$ and {\color{black}  a smooth path $(X_t)_{t\in(-\epsilon,\epsilon)}$}  in $T(S)$  with $X_0 = X$. The tangent vector $\dot{X}_0:=\left.\frac{d}{dt}\right|_{t=0}X_t$ can be uniquely represented by a harmonic Beltrami differential
$$\dot{X}_0 = \beta^{-2}\overline{q}\textcolor{black}{,}$$
where $\beta$ is the hyperbolic metric of $X_0$ and $q$ is a holomorphic quadratic differential on $X_0$. In a local coordinate chart $z$, we can write $\beta = \beta(z)|dz|$, where $|dz| =|d(x+iy)|= \sqrt{dx^2+dy^2}$, and $q = q(z)dz^2$ so that $\bar{q} = \overline{q(z)}d\bar{z}^2$. Hence, we have $\dot{X}_0 = \frac{\overline{q(z)}d\bar{z}^2}{\beta^2(z)dzd\bar{z}}= \frac{\overline{q(z)}}{\beta^2(z)}\frac{d\bar{z}}{dz}$.

The {\it Weil-Petersson metric} on $T(S)$ is given by
$$||\dot{X}_0||^2_{WP} := \frac{\int_{X_0}\beta^2|\beta^{-2} \overline{q}|^2}{4\pi(g-1)}= \frac{\int_{X_0} \beta^{2}\beta^{-4}|q|^2}{4\pi(g-1)} = \frac{\int_{X_0} \beta^{-2}|q|^2}{4\pi(g-1)}.$$
In \cite{McMullen08}, McMullen proved that the Weil-Petersson metric is a constant multiple of the second derivative of the Hausdorff dimension function arising from {\it matings of Fuchsian groups}, which we discuss below.

\subsection*{Matings of Fuchsian groups}
Fix a smooth path $(X_t)_{t\in(-\epsilon,\epsilon)}$  in $T(S)$. {\color{black} We denote by $(G_t)$ the corresponding smooth family of Fuchsian groups such that $X_t = \Delta/G_t$ for every $t$, where the unit disk $\Delta$ is regarded as the hyperbolic plane $\mathbb H^2$.}

Let us identify the unit circle $\mathbb{S}^1$ with $\partial \mathbb H^2$. Then for each $t \in (-\epsilon,\epsilon)$, there is a unique quasi-symmetric homeomorphism $h_t \from \mathbb S^1 \to \mathbb S^1$ isotopic to the identity map such that $h_t$ conjugates the $G_0$-action to the $G_t$-action and satisfies $h_0(z) = z$.

\textcolor{black}{Let us denote by $1/\Delta$ the set $\{z \in \widehat{\mathbb C}: |z|>1\}$.} For each $t\in(-\epsilon,\epsilon)$, we glue $(\Delta,G_0)$ and $(1/\Delta,G_t)$ along their boundary circles via $h_t$ and obtain a {\it quasi-Fuchsian group} $\Gamma_t$. {\color{black} More precisely, for the quasi-symmetric map $h_t \from \mathbb S^1 \to \mathbb S^1$, there exists a quasi-circle $J_t$ in $\widehat{\mathbb{C}}$ whose boundary correspondence is $h_t$; see \cite[Theorem 3]{McM_simul}. We glue $(\Delta,G_0)$ and $(1/\Delta,G_t)$ via the uniformization map of the components of $\widehat{\mathbb{C}}\setminus J_t$, and obtain a Kleinian group $\Gamma_t$.}
\textcolor{black}{For all $t$,} the limit set $\Lambda(\Gamma_t)$ is a quasi-circle \textcolor{black}{  such that $\Lambda(\Gamma_0) = \mathbb S^1$}. We denote by ${\rm H.dim}(A)$ the Hausdorff dimension of a compact subset $A \subset \widehat{\mathbb C}$ with respect to the spherical metric. As a function of $t$, the Hausdorff dimension $\delta(\Gamma_t):={\rm H.dim}(\Lambda(\Gamma_t))$ is real-analytic and minimized at $t = 0$.

The Hausdorff dimension of a measure $\mu$ {\color{black} on $\mathbb{S}^1$} is defined by
\begin{equation}\label{eq:H.dim of Measure}
    \mathrm{H.dim}(\mu) := \inf \left\{ {\rm H.dim}(E) ~|~ \mu(\mathbb{S}^1 \setminus E)=0,~ E\subset \mathbb{S}^1\right\}.
\end{equation}
Denote by $m_{\rm leb}$ the Lebesgue probability measure on $\mathbb{S}^1$. We define $m_t:=(h_t)_*m_{\rm leb}$ as the push-forward of $m_{\rm leb}$ by the map $h_t$. As a function of $t$, the Hausdorff dimension ${\rm H.dim}(m_t)$ of $m_t$ is maximized at $t = 0$.

\begin{thm}[{\cite[Theorem 1.1]{McMullen08}}]\label{thm:McMullen1}
For a smooth family $(G_t)$ of Fuchsian groups, we have
\begin{align*}
\frac{d^2}{dt^2}\bigg|_{t=0} \delta(\Gamma_t) = -\frac{1}{4}\frac{d^2}{dt^2}\bigg|_{t=0} {\rm H.dim}(m_t) = \frac{1}{3}  \frac{||\dot{X}_0||^2_{WP}}{4\pi(g-1)},
\end{align*}
where $\dot{X}_0\in T_{X_0}(S)$ is the tangent vector at $t=0$ of the smooth path $(X_t=\Delta/G_t)$ in $T(S)$.
\end{thm}

\subsection*{Holomorphic vector fields}
Consider the quasi-Fuchsian group $\Gamma_t$ that is obtained by mating the Fuchsian groups $G_0$ and $G_t$ as described above. There is a smooth family of {\color{black} univalent} maps on the unit disk $H_t\from \Delta \to \widehat{\mathbb{C}}$ so that $\Gamma_t \circ H_t = H_t \circ \Gamma_0$. Recall that $\Gamma_0$ acts on the unit disk $\Delta \cong \mathbb H^2$.

For each fixed $z \in \Delta$, $\{H_t(z)\}_{t \in (-\epsilon,\epsilon)}$ is a smooth path in $\widehat{\mathbb C}$. Note that $H_0(z)=z$. Hence, the derivative $\frac{d}{dt}\big|_{t=0}H_t(z)$ is a tangent vector in $T_{z}\Delta$. Therefore
\begin{equation}\label{eq_hvf}
v(z) :=\left.\frac{d}{dt}\right|_{t=0}H_t(z)
\end{equation}
defines a holomorphic vector field on $\Delta$, which encodes the infinitesimal deformation of $(G_t)$ at $t=0$. Note that $H_t$ is a holomorphic conjugacy between the two dynamical systems $(\Delta,\Gamma_0)$ and $(H_t(\Delta), \Gamma_t)$, both of which are also holomorphically conjugate to $(\Delta, G_0)$.

{\color{black} For each fixed $z \in \Delta$, we identify the tangent space $T_z\Delta$ with $\mathbb C$. The vector field $v$ in Equation \eqref{eq_hvf} can also be regarded as a function $v \from \Delta \to \mathbb C$ given by $v(z)=\left.\frac{d}{dt}\right|_{t=0}H_t(z) \in \mathbb {C}$. Hence, we can consider its derivative $v'(z)$.
}

\begin{thm}[{\cite[Theorem 1.3]{McMullen08}}]\label{thm:McMullen2}
For a smooth family $(G_t)$ of Fuchsian groups, we have
\[
    \frac{d^2}{dt^2}\bigg|_{t=0} \delta(\Gamma_t) =\lim_{r\to1}\frac{1}{4\pi \log{|1-r|}} \int_{|z|=r} |v'(z)|^2\,|dz|.
\]
\end{thm}

The first equality in Theorem \ref{thm:McMullen1} 
follows from the thermodynamic formalism that will be discussed in the next paragraphs. The proofs of the second equality in Theorem \ref{thm:McMullen1} and Theorem \ref{thm:McMullen2} are more involved. McMullen developed a theory of holomorphic forms of foliated unit tangent bundles and of Riemann surface laminations, which we do not discuss in this article. The relationship between the pressure metric and the Weil-Petersson metric can also be obtained by Bridgeman-Taylor's work (see Theorem \ref{thm:BridgemanTaylor} below), which relies on a result by Wolpert \cite{Wolpert_WPmetricbyThurston}. As stated in \cite[Theorem 12.1]{McMullen08}, theorems in \cite{McMullen08} give an alternative proof of Wolpert's result.

\subsection*{Thermodynamic formalism for Fuchsian groups}
Fix a smooth path $(X_t)_{t \in (-\epsilon,\epsilon)}$ in $T(S)$. Denote by $(G_t)$ the smooth family of Fuchsian groups with $X_t = \Delta/G_t$ for each $t \in (-\epsilon,\epsilon)$.

By \cite{Bowen79_Markov}, to each $X_t$ there is an associated expanding Markov map $f_t \from \mathbb S^1 \to \mathbb S^1$, i.e., $\mathbb S^1$ is decomposed into finitely many arcs so that $f_t$ restricted to each arc is the action by some element of $G_t$ and maps each arc onto a union of arcs. Moreover, this (topological) decomposition {\color{black} can be chosen to be} preserved under the deformation along $(X_t)$. More precisely, there is a family of homeomorphisms $h_t \from \mathbb S^1 \to \mathbb S^1$ conjugating the dynamics of $f_0$ to the dynamics of $f_t$ so that $h_t$ also conjugates the $G_0$-action to the $G_t$-action on $\mathbb S^1$.

Gluing $(\Delta,G_s)$ to $(\Delta,G_t)$ using $h_t \circ h_s^{-1}$, we obtain a two-parameter family of quasi-Fuchsian groups $(\Gamma_{s,t})$ and the corresponding Markov maps
$(F_{s,t}\from \Lambda(\Gamma_{s,t}) \to \Lambda(\Gamma_{s,t}))$, which are piece-wisely defined by the restrictions of elements of $\Gamma_{s,t}$. The family $(F_{s,t})$ can be normalized so that $F_{t,t}(z)=f_t(z)$ and 
\begin{equation}\label{eqn:symm}
    F_{s,t}(\overline{z}) = \overline{F_{t,s}(z)}.
\end{equation}
Using the Markov partition and the theory of holomorphic motions, we have a smooth family of symbolic coding
$$\pi_{s,t} \from\Sigma_A^+ \to \Lambda(\Gamma_{s,t}).$$
Here, $\Sigma_A^+$ denotes the one-sided subshift of finite type corresponding to the expanding Markov maps $f_t$. {\color{black} The matrix $A$ is obtained via the Markov partition $\{U_j\}$ of $f_t$; namely $A_{ij}=1$ if $U_j \subset f_t(U_i)$ and $A_{ij}=0$ otherwise.}

Define $\phi_t \from \Sigma_A^+ \to \mathbb{R}$ and $\Phi_t \from \Sigma_A^+ \to \mathbb{R}$ by
$$\phi_t(\underline{x}) := -\log|F_{0,t}'(\pi_{0,t}(\underline{x}))|
\text { and } \Phi_t(\underline{x}) := -\log|F_{t,t}'(\pi_{t,t}(\underline{x}))|.$$
Denote by $\delta(s,t)$  the Hausdorff dimension of the limit set $\Lambda(\Gamma_{s,t})$. Note that we have $\Pcal(\delta(0,t) \phi_t) = 0$ and
$\Pcal(\delta(t,t)  \Phi_t) = 0$ by Bowen's result \cite{Bowen79_Markov}. Denote by $m_0 = m(\phi_0)$ the equilibrium measure of $\phi_0$.
Note that $m_0$  \textcolor{black}{is equivalent to the Lebesgue measure;  see \cite[Theorem 2.3]{McMullen08}. Recall that two measures $\mu$ and $\nu$ are equivalent if $\mu(A)=0\iff \nu(A)=0$ for any measurable subset $A$.} 

\medskip

For any $X\in T(S)$, we have a Markov expanding map $f_X\from\mathbb{S}^1\to \mathbb{S}^1$. Moreover, the maps $f_X\from\mathbb{S}^1\to \mathbb{S}^1$ and $f_Y\from\mathbb{S}^1\to \mathbb{S}^1$ are conjugate for any $X,Y \in T(S)$. Therefore, {\color{black} the space} $\Sigma_A^+$ defined above \textcolor{black}{encodes} the (topological) dynamics of $f_X$ on $\mathbb{S}^1$ for any $X \in T(S)$. {\color{black} We denote by $\pi_X \from \Sigma_A^+ \to \mathbb S^1$ the projection map.} Define $\Phi_X \from \Sigma_A^+ \to \mathbb{R}$ by
\[
    \Phi_X(\underline{x}):= -\log |f'_X(\pi_X(\underline{x}))|.
\]
It turns out that $\Phi_X$ is H\"{o}lder continuous.
Then we define the thermodynamic map
$$\mathcal{E} \from T(S) \to \mathcal{C}(\Sigma^+_A)$$ given by
$$X \mapsto [\Phi_X].$$
We note that $\mathcal{P}(\Phi_X) = 0$ again by Bowen's result \cite{Bowen79_Markov}.
For a smooth family $(X_t)_{t\in(-\epsilon,\epsilon)}$, 
The pullback of the pressure norm $||\cdot||_{pm}$ on $\mathcal{C}(\Sigma_A^+)$ to $T(S)$ via $\mathcal{E}$ is given by
\begin{equation*}\label{eq_2.21}
\left|\left|\left.\frac{d}{d t}\right|_{t=0}X_t\right|\right|^2_\Pcal:=
     \frac{{\rm Var}(\dot{\Phi}_0, m_0)}{-\int \phi_0 \, dm_0}.
\end{equation*}

McMullen proved the equality
\begin{equation*}
 \left.\frac{d^2}{dt^2}\right|_{t=0} {\rm H.dim}~ \Lambda(\Gamma_{0,t}) = \frac{{\rm Var}(\dot{\phi}_0, m_0)}{-\int \phi_0 \, dm_0}
\end{equation*}
in \cite[Theorem 2.7]{McMullen08}. Since $\Gamma_{0,t}=\Gamma_t$, the above equation becomes
\begin{equation}\label{eq_2.22}
\frac{d^2}{dt^2}\bigg|_{t=0} \delta(\Gamma_t) = \frac{{\rm Var}(\dot{\phi}_0, m_0)}{-\int \phi_0 \, dm_0}.
\end{equation}
On the other hand, we have (see \cite[Proof of Theorem 2.6]{McMullen08})
\begin{equation}\label{eq_2.23}
    \frac{1}{4}\frac{{\rm Var}(\dot{\Phi}_0, m_0)}{-\int \phi_0 \, dm_0} = \frac{{\rm Var}(\dot{\phi}_0, m_0)}{-\int \phi_0 \, dm_0}
\end{equation}
and
\begin{equation}\label{eq_2.24}
\frac{d^2}{dt^2}\bigg|_{t=0} {\rm H.dim}(m_t) = -\frac{{\rm Var}(\dot{\Phi}_0, m_0)}{-\int \phi_0 \, dm_0}. 
\end{equation}
Therefore Equations (\ref{eq_2.22}), (\ref{eq_2.23}), and (\ref{eq_2.24}) give the first equality in Theorem \ref{thm:McMullen1}.


\subsubsection{Pressure  forms and semi-norms on the spaces of Blaschke products} \label{sec_Curt_2}
Fix $d \ge 2$. Recall that $\mathcal{B}_d$ is the quotient space of degree-$d$ proper holomorphic self-maps of the unit disk $\Delta$ by the conjugate action of $\Aut_{\mathbb{C}}(\Delta)$.
Any $[f] \in \mathcal{B}_d$ can be represented by a finite Blaschke product
\begin{equation}\label{eq:blas}
    f(z) = z \prod_{i=1}^{d-1}\frac{z-a_i}{1-\overline{a_i}z},
\end{equation}
where the $a_i$'s are in $\Delta$.
The Julia set $\Julia(f)$ of $f$ is the unit circle $\mathbb S^1$, and the map $f\from \mathbb S^1 \to \mathbb S^1$ is expanding.

\subsection*{Matings of Blaschke products}
Fix a smooth path $(f_t)_{t\in(-\epsilon,\epsilon)}$ in $\mathcal{B}_d$. There is a unique isotopy $h_t\from \mathbb S^1 \to \mathbb S^1$ conjugating the dynamics of $f_0$ to $f_t$ on $\mathbb S^1$ and satisfying $h_0(z) = z$. Using $h_t$ to glue $(\Delta,f_0)$ and $(1/\Delta,f_t)$ along $\mathbb S^1$, we obtain a smooth family of rational maps
$$F_t \from \widehat{\mathbb C} \to \widehat{\mathbb C}.$$
The Julia set $\Julia(F_t)$ is a quasi-circle with $\Julia(F_0) = \mathbb S^1.$ Thus the Hausdorff dimension ${\rm H.dim}(\Julia(F_t))$ is minimized at $t = 0$.

We define $m_t :=(h_t)_*m_{\rm leb}$ as the push-forward of the Lebesgue probability measure $m_{\rm leb}$ on $\mathbb{S}^1$ by the map $h_t$. Then the Hausdorff dimension $\mathrm{H.dim}(m_t)$ of $m_t$ defined as \eqref{eq:H.dim of Measure} is maximized at $t=0$.

The definition of the holomorphic vector field $v$ on the unit disk $\Delta$ is the same as that for Fuchsian groups; see Equation \eqref{eq_hvf}. The thermodynamic formalism described in the previous subsection also applies to Blaschke products. For $\dot{f}_0:=\left.\frac{d}{d t}\right|_{t=0}f_t$, we define the {\it pressure  {\color{black} semi-norm} for Blaschke products} by
\[
    ||\dot{f}_0||^2_\Pcal:=\frac{{\rm Var}(\dot{\Phi}_0, m_0)}{-\int \phi_0 \, dm_0}.
\]

\begin{thm}[{\cite[Theorems 1.6, 1.7, and 2.6]{McMullen08}}]\label{thm:McMullen4}
For a smooth family of degree-$d$ Blaschke products $(f_t)$, we have
\[
    \frac{d^2}{dt^2}\bigg|_{t=0} \delta(F_t) =\lim_{r\to1}\frac{1}{4\pi \log{|1-r|}} \int_{|z|=r} |v'(z)|^2\,|dz| = -\frac{1}{4} \frac{d^2}{dt^2}\bigg|_{t=0} {\rm H.dim}(m_t)=\frac{1}{4} \left|\left| \dot{f}_0\right|\right|^2_\Pcal,
\]
where $\dot{f}_0$ is the tangent vector of the path $(f_t)$ in $\Bcal_d$ at $t=0$.
\end{thm}

Proofs of the identities in Theorem \ref{thm:McMullen4} are analogous to those for Fuchsian groups. Motivated by Theorems \ref{thm:McMullen1} and \ref{thm:McMullen4}, McMullen defined the {\it Weil-Petersson semi-norm on $\Bcal_d$} by
\[
    ||\dot{f}_0||^2_{WP}:=\frac{d^2}{dt^2}\bigg|_{t=0} \delta(F_t).
\]
Recall that the semi-norm $||\cdot||_{WP}$ defines a path pseudo-metric $d_{WP}$, see Section \ref{sec:ThermodynFormalism}. Ivrii proved the pseudo-metric space $(\Bcal_2, d_{WP})$ is incomplete \cite{Ivrii}. We conjecture that $||\cdot||_{WP}$ is a norm and $d_{WP}$ is a metric.



\subsection{Pressure forms on quasi-Fuchsian spaces}\label{sec:Bridgeman}
In this subsection, we survey the results in \cite{BridgemanTaylor,Bridgeman_WPMetricQF} on the pressure forms on quasi-Fuchsian spaces. {\color{black} In this case, the pressure forms are not positive definite. Theorem \ref{thm:BridgemanTaylor} gives a characterization of their degeneracy loci.}

Fix a closed orientable surface $S$ of genus $g\ge 2$. Denote by $T(S)$ the Teichm\"uller space of $S$. Fix a Fuchsian group $\Gamma_0$ that is isomorphic to $\pi_1 S$, i.e., $\Gamma_0:=\rho(\pi_1S)$ where $\rho \in T(S)$. 
The {\it quasi-Fuchsian space} $\mathcal{QF}(S)$ of $S$ is the set of equivalence classes $X = [(f ,\Gamma )]$ where $f\from  \partial_{\infty} \mathbb H^3 \to \partial_{\infty} \mathbb H^3$ is a quasi-conformal homeomorphism of the boundary of the hyperbolic $3$-space $\mathbb H^3$
conjugating $\Gamma_0$ to $\Gamma${\color{black}, i.e., such that} $\Gamma = f\circ \Gamma_0 \circ f^{-1}$. Then $\Gamma$ is a torsion-free Kleinian group, i.e., a discrete torsion-free subgroup of $\PSL(2,\mathbb C)$, which acts on $\mathbb H^3 \cup \partial_\infty \mathbb H^3$.
We say that $(f_1,\Gamma_1)$ and $(f_2, \Gamma_2)$ are {\it equivalent} if there exists a conformal automorphism $\alpha\from  \partial_{\infty} \mathbb H^3 \to \partial_{\infty} \mathbb H^3$ conjugating $\Gamma_1$ to $\Gamma_2${\color{black}, i.e., such that} $f_2 \circ \gamma \circ f_2^{-1} = (\alpha \circ f_1) \circ \gamma \circ (\alpha \circ f_1)^{-1}$ for any $\gamma\in\Gamma_0$. We refer the interested reader to \cite{Maskit} for the theory of Kleinian groups.

Identifying the orientation-preserving isometry group of  $\mathbb H^3$ with $\PSL(2,\mathbb C)$, we obtain a complex structure on $\mathcal{QF}(S)$ and the induced almost-complex structure $J \from T(\mathcal{QF}(S)) \to T(\mathcal{QF}(S))$ for the quasi-Fuchsian space $\mathcal{QF}(S)$. Here $T(\mathcal{QF}(S))$ denotes the tangent bundle of $\mathcal{QF}(S)$. By Bers simultaneous uniformization theorem \cite{Bers_SimultaneousUniform}, $\mathcal{QF}(S)$ is biholomorphic to $T(S) \times T(\overline{S})$
where $\overline{S}$ has the opposite orientation to $S$.
There is a natural diagonal embedding
$$i \from T(S) \to \mathcal{QF}(S) = T(S) \times T(\overline{S})$$
given by $i(X) = (X, \overline{X})$. The image $\mathcal{F}(S):=i(T(S))$ is the Fuchsian locus in $\mathcal{QF}(S)$ and diffeomorphic to $T(S)$.

Denote by $||\cdot ||_{WP}$ the Weil-Petersson metric on $T(S)$. 
We call
$
    \frac{2}{3\pi |\chi(S)|}||\cdot ||_{WP}
$
the {\it normalized Weil-Petersson metric} where $\chi(S)$ is the Euler characteristic of the surface $S$.

\subsection*{Thermodynamic formalism and pressure  {\color{black} forms} for quasi-Fuchsian manifolds}
In \cite{Bowen79_Markov}, Bowen established the thermodynamic formalism for quasi-Fuchsian groups. In particular, every quasi-Fuchsian group $\Gamma$ admits an {\it expanding Markov map} $f_{\Gamma} \from \Lambda(\Gamma) \to \Lambda(\Gamma)$ on its limit set $\Lambda(\Gamma)$. Recall that the limit set $\Lambda(\Gamma)$ of $\Gamma$ is defined as the set of accumulation points on $\partial_\infty \Hbb^3$ of the orbit $\Gamma \cdot a$ for $a\in \Hbb^3$. The limit set $\Lambda(\Gamma)$ is independent of the choice of $a\in \Hbb^3$.
Denote by $(\Sigma_A^+, \sigma)$ the associated subshift of finite type. Denote by $\pi_\Gamma \from \Sigma_A^+ \to \Lambda(\Gamma)$ the projection map. 

\textcolor{black}{Fix} $X_1 = [(f_1,\Gamma_1)] \in \mathcal{QF}(S)$.
If $f_1 \from \mathbb S^2 \to \mathbb S^2$ is a quasiconformal map conjugating $\Gamma$ and $\Gamma_1$, then the map $f_{\Gamma_1} \from \Lambda(\Gamma_1) \to \Lambda(\Gamma_1)$ given by
$$f_{\Gamma_1} = f_1 \circ f_{\Gamma} \circ f_1^{-1}$$ is an expanding Markov map for $\Gamma_1$ with the same symbolic coding $(\Sigma_A^+, \sigma)$. \textcolor{black}{$\pi_{\Gamma_{1}}$ and $\pi_{\Gamma}$ satisfy $\pi_{\Gamma_1} = f_1 \circ \pi_{\Gamma}$.}

Define $\phi_{\Gamma} \from \Sigma_A^+ \to \mathbb R$ by
$$\phi_{\Gamma}(\underline{x}) := - \log |f_{\Gamma}'(\pi_{\Gamma}(\underline{x}))|.$$
Then $\phi_{\Gamma}$ is H\"older continuous. 
If $[(f_1,\Gamma_1)] = [(f_2,\Gamma_2)]$, then $\phi_{\Gamma_1}$ is cohomologous to $\phi_{\Gamma_2}$. Also by Bowen's result \cite{Bowen79_Markov}, we have
$\Pcal(\delta(\Gamma_1) \phi_{\Gamma_1}) = 0.$
Therefore the map
$\mathcal{E} \from \mathcal{QF}(S) \to \mathcal{C}(\Sigma_A^+)$ given by
$$X \mapsto [\delta(X) \phi_{X}]$$
is well-defined.

Given $X \in \mathcal{QF}(S)$ and $v \in T_X\mathcal{QF}(S)$, choose a smooth curve $\alpha \from(-\epsilon,\epsilon) \to \mathcal{QF}(S)$ with $\alpha(0) = X$ and $\left.\frac{d}{dt}\right|_{t=0}\alpha(t) = v$. Then the pullback of the pressure norm is given by
$$||v||^2_{\Pcal} := \left|\left| \left.\frac{d}{dt}\right|_{t=0} \mathcal{E}(\alpha(t)) \right|\right|^2_{pm} = \frac{\mathrm{Var}(\psi,m(\phi))}{-\int \phi\, dm(\phi)}$$
where $\phi = \mathcal{E}(\alpha(0))$, $\psi = \left.\frac{d}{dt}\right|_{t=0} \mathcal{E}(\alpha(t))$, and $m(\phi)$ is the equilibrium measure for $\phi$. {\color{black}  Since the pullback may not preserve the non-degeneracy, $||\cdot ||_\Pcal$ is a semi-norm in general. We call $||\cdot ||_\Pcal$ the {\it pressure semi-norm} on $T_X \QF(S)$. Similarly, by pulling back $\langle\cdot,\cdot\rangle_{pm}$, we obtain a positive semi-definite symmetric bilinear 2-form $\langle\cdot,\cdot\rangle_\Pcal$ on $\QF(S)$. We call it the {\it pressure form} on $\QF(S)$.}

In \cite{BridgemanTaylor, Bridgeman_WPMetricQF}, Bridgeman and Taylor also constructed another  {\color{black}semi-norm} $||\cdot||_G$ by using the so-called length functions that we introduce in the subsequent paragraphs. The semi-norms $||\cdot||_G$ and $||\cdot||_\Pcal$ are conformally equivalent; see Theorem \ref{thm:BridgemanTaylor}.

\subsection*{Patterson-Sullivan geodesic currents and length functions}
Denote by
\[
    \mathcal{G}(\mathbb H^3) \cong \left(\partial_\infty \mathbb H^3 \times \partial_\infty \mathbb H^3 \setminus ({\rm diagonal})\right)/\mathbb Z_2
\]
the space of unoriented geodesics in $\mathbb H^3$. Here $\mathbb Z_2$ acts by exchanging of two coordinates of $\partial_\infty \mathbb H^3 \times \partial_\infty \mathbb H^3$, 
i.e., $(a,b)\mapsto (b,a)$ for $a,b\in \partial_\infty \mathbb H^3$.

Suppose that $\Gamma$ is a Kleinian group.  A {\it geodesic current} for $\Gamma$ is a positive measure on $\mathcal{G}(\mathbb H^3)$ that is invariant under the action of $\Gamma$ and supported on the set of geodesics whose endpoints are in the limit set $\Lambda(\Gamma)$.

For $s > 0$, the {\it Poincar\'e series} of a Kleinian group $\Gamma$ is defined by
$$g_s(x,y) := \sum_{\gamma \in \Gamma} e^{-s\cdot d(x,\gamma \cdot y)},$$
where $x,y \in \mathbb H^3$ and $d(\cdot,\cdot)$ is the hyperbolic distance on $\mathbb H^3$. The {\it critical exponent} $\delta_\Gamma$ is defined as
$$\delta_\Gamma := \inf \{s~|~ g_s(x,y) < \infty\}$$
and is independent of the choices of $x$ and $y$.

Using the Poincar\'e series, \textcolor{black}{a} {\it Patterson-Sullivan measure} is constructed as follows. For $x,y \in \mathbb H^3$ and $s > \delta_\Gamma$, we define a measure $\mu_{x,s}$ supported on the orbit of $y$ by
$$\mu_{x,s} = \frac{1}{g_s(y,y)} \sum_{\gamma \in \Gamma} e^{-s\cdot d(x,\gamma \cdot y)} \cdot  \delta_{\gamma \cdot y}$$
where $\delta_p$ is the Dirac mass at the point $p\in \mathbb H^3$. \textcolor{black}{A} {\it Patterson-Sullivan measure} $\mu_x$ is a weak$^*$-limit of the measures $\mu_{x,s}$ as $s\searrow \delta_{\Gamma}$.  
Then we have ${\rm supp}(\mu_x)\subset \Lambda(\Gamma)$. Moreover, $\mu_x$ is a {\it conformal density of dimension $\delta_\Gamma$}, i.e.,
\[
    \mu_x(\gamma E)=\int_E |\gamma'|^{\delta_\Gamma} d\mu_x
\]
for every Borel set $E\subset \partial_\infty \Hbb^3$ and $\gamma \in \Gamma$.

We define a measure $\tilde{m}$ on $\left(\partial_\infty \mathbb H^3 \times\partial_\infty \mathbb H^3 \setminus ({\rm diagonal})\right)$ by
$$d\tilde{m}(a,b) = \frac{d\mu_x(a)d\mu_x(b)}{|b-a|^{2\delta_\Gamma}}.$$
Define a projection map $\pi \from \left(\partial_\infty \mathbb H^3 \times\partial_\infty \mathbb H^3 \setminus ({\rm diagonal})\right) \to \mathcal{G}(\mathbb H^3)$ sending two distinct points on $\partial_\infty \mathbb H^3$ to the unoriented geodesic in $\mathbb H^3$ connecting them. Then the measure $m := \pi_\ast(\tilde{m})$ is $\Gamma$-invariant and supported on $\left(\Lambda(\Gamma) \times \Lambda(\Gamma) \setminus ({\rm diagonal}) \right)/\mathbb Z_2$. Hence $m$ is a geodesic current, and we call it  \textcolor{black}{a} {\it Patterson-Sullivan geodesic current} for $\Gamma$.

If $\Gamma$ is geometrically finite {\color{black}(i.e., $\Gamma$ has a polyhedral fundamental domain with finitely many sides)}, then $\mu_x$ is independent of $x$ and unique up to constant multiple \cite{Sullivan84}. Hence, $\tilde{m}$ and $m$ are also unique up to constant multiple. In this case, we call $m$ the {\it unit length Patterson-Sullivan geodesic current} if it is a probability measure.

Suppose that $\Gamma$ is a convex cocompact Kleinian group. Given $\gamma \in \Gamma$, the {\it length of $\gamma$ with respect to $\Gamma$} is its translation length, namely, it is equal to
{\color{black} $2\log |\lambda|$ where $\lambda$ and $1/\lambda$ are the eigenvalues of the (loxodromic) matrix $\gamma$ in $\PSL(2,\mathbb C)$.}
We note that the translation length of $\gamma$ is also equal to the hyperbolic length of the closed geodesic that $\gamma$ represents.
The notion of length for $\gamma$ (or closed geodesics) can be extended to the notion of {\it length of a geodesic current with respect to $\Gamma$}; see \cite[Proposition 4.5]{Bonahon_EndsHyp3Mfd} and \cite[Proposition 14]{Bonahon_Currents}. 

For a convex cocompact Kleinian group $\Gamma$, we define the space $\mathcal{QC}(\Gamma)$ of quasi-conformal deformations of $\Gamma$ as the set of equivalence classes $X=[(f_X,\Gamma_X)]$ with $\Gamma_X=f_X\circ \Gamma \circ f_X^{-1}$. 
Given a geodesic current $\alpha$ for $\Gamma$, we define the length function $L_\alpha \from \mathcal{QC}(\Gamma) \to \mathbb R$ by sending $[(f_Y,\Gamma_Y)]$ to the length of the geodesic current $(f_Y)_*\alpha$ with respect to $\Gamma_Y$. We refer the reader to \cite[Section 3]{BridgemanTaylor} for details.

\subsection*{Symmetric bilinearsa 2-form $\langle\cdot,\cdot\rangle_G$}
Fix $X_0 \in \QF(S)$. Denote by $\mu_0$ the unit length Patterson-Sullivan geodesic current of $X_0$. 
We define $G_{X_0} \from \QF(S) \to \Rbb$ by
    \[
        [f_Y,\Gamma_Y] \mapsto \delta([(f_Y,\Gamma_Y)]) L_{\mu_0}([f_Y,\Gamma_Y]),
    \]
where $\delta \from \QF(S) \to [0,2]$ is the function sending $[(f,\Gamma)]$ to the Hausdorff dimension of its limit set $\Lambda(\Gamma)$ and $L_{\mu_0}([f_Y,\Gamma_Y])$ is the length of $(f_Y)_*\mu_0$ with respect to $\Gamma_Y$. 
     


\begin{thm}[{\cite{BridgemanTaylor,Bridgeman_WPMetricQF}}]\label{thm:BridgemanTaylor}
Fix a closed orientable surface $S$ of genus $g\ge2$.
\begin{enumerate}
    \item 
    The map $G_{X_0}\from \QF(S) \to \Rbb$ is real-analytic and has a unique global minimum at $X_0$. Hence the Hessian of $G_{X_0}$ defines a \textcolor{black}{positive semi-definite symmetric bilinear $2$-form $\langle\cdot,\cdot\rangle_G$.}
    \item 
    The 2-form $\langle\cdot,\cdot\rangle_G$ restricted to the Fuchsian locus $\mathcal{F}(S)$ is equal to the normalized Weil-Petersson metric. More precisely, for any $v,w \in T_{X_0}(\mathcal{F}(S)) \subset T_{X_0}(\mathcal{QF}(S))$, we have
    \begin{equation}\label{eq_WPext}
        \langle v,w \rangle_G = \frac{2}{3\pi |\chi(S)|}\langle v,w \rangle_{WP} .
    \end{equation}
    \item For any $X_0\in \QF(S)$ and $v\in T_{X_0}(\QF(S))$, we have
    \[
        || v ||_G=\sqrt{\delta(X_0)} || v ||_\Pcal.
    \]
    \item For $X_0\in \QF(S)$ and $v\in T_{X_0}(\QF(S))$, we have $||v||_G=0$ if and only if
    \begin{enumerate}
        \item $X_0\in \mathcal{F}(S)$, and
        \item $v \in J\cdot T_{X_0}(\mathcal{F}(S))$ where $J$ is the almost complex structure on $\QF(S)$.
    \end{enumerate}
\end{enumerate}
\end{thm}

{\color{black}
\subsection*{Pure shearing and pure bending tangent vectors}
Fix a closed orientable surface $S$ of genus at least $2$. By Bers simultaneous uniformization \cite{Bers_SimultaneousUniform}, we have a biholomorphism $\QF(S) = T(S) \times T(\overline{S})$,
where $\overline{S}$ is the surface $S$ with opposite orientation. Consider the diagonal embedding $i\from T(S) \to \QF(S)$, $i(X)=(X,\overline{X})$ for $X\in T(S)$. Set $\mathcal{F}(S):=i(T(S))$.

Fix $\Gamma \in \mathcal{F}(S)$ and denote by $J$ the (almost) complex structure on $\QF(S)$. Since $\mathcal{F}(S)$ is the locus of fixed points of the anti-holomorphic involution $\iota \colon \QF(S) \to \QF(S)$ defined by $(\tau_1,\overline{\tau_2}) \mapsto (\tau_2,\overline{\tau_1})$, the tangent space of $T_\Gamma \mathcal{QF}(S)$ can be decomposed as follows:
\begin{equation}\label{eqn:DecompTQF}
    T_\Gamma(\mathcal{QF}(S))= T_\Gamma(\mathcal{F}(S)) \oplus J \cdot T_\Gamma(\mathcal{F}(S)).
\end{equation}
Tangent vectors in $T_\Gamma(\mathcal{F}(S))$ are called {\em pure shearing} because the corresponding deformations of hyperbolic structures of surfaces can be understood as shearing along laminations \cite{Thu_Earthquakes}.
Tangent vectors in $J \cdot T_\Gamma(\mathcal{F}(S))$ are called {\em pure bendings} because they can be seen as the deformations of Fuchsian groups by bending their convex cores, which are isometrically embedded $\mathbb{H}^2$ in $\mathbb{H}^3$, along geodesic laminations. By Theorem \ref{thm:BridgemanTaylor}, the pure bending vectors are the only tangent vectors at which the pressure form degenerates.
}

\subsection{Pressure metrics on deformation spaces of Anosov representations}\label{sec:AnosovReps}

Anosov representations were introduced by Labourie \cite{Labourie06} as generalizations of representations of surface groups into Lie groups of rank one to Lie groups of higher rank. The theory then was generalized to representations of any word hyperbolic groups into Lie groups of higher rank by Guichard-Wienhard \cite{GW11}. In \cite{BCLS}, Bridgeman-Canary-Labourie-Sambarino constructed pressure metrics on deformation spaces of Anosov representations. We summarize the results in \cite{BCLS} in this subsection.

Suppose that $\Gamma$ is a word hyperbolic group. One example of such a group is the fundamental group of a closed hyperbolic manifold. Fix an integer $d \ge 2$. Given an integer $p$ with $1 \le p \le d-1$, we denote by $\mathcal{G}_p(\mathbb R^d)$ the Grassmannian of $p$-dimensional vector subspaces of $\mathbb R^d$. A homomorphism $\rho \from \Gamma \to \PSL(d,\mathbb R)$ is {\it $a_p$-Anosov} if there exist a pair of $\rho$-equivariant H\"older continuous maps $(\zeta_\rho^p, \zeta_\rho^{d-p}) \from \partial_\infty \Gamma \to \mathcal{G}_p(\mathbb R^d) \times \mathcal{G}_{d-p}(\mathbb R^d)$ such that for any $x, y \in \partial_\infty \Gamma$ with $x \neq y$, we have
$$\zeta_\rho^p(x) \oplus \zeta_\rho^{d-p}(y) = \mathbb R^d,$$
and a suitable associated flow is contracting; see \cite{BCLS} for details. {\color{black} Here the $\rho$-equivariance means $(\zeta_\rho^p, \zeta_\rho^{d-p})(\gamma\cdot x) = (\rho(\gamma) \cdot \zeta_\rho^p(x), \rho(\gamma) \cdot\zeta_\rho^{d-p}(x))$ for all $\gamma \in \Gamma$ and $x \in \partial_\infty \Gamma$.}
An $a_1$-Anosov representation is also called a {\it projective Anosov} representation.

A nice class of Anosov representations is given by {\it Hitchin representations}. Fix a closed orientable surface $S$ of genus $g \ge 2$. Denote by $\pi_1(S)$ its fundamental group. A homomorphism $\rho \from \pi_1(S) \to \SL(d,\mathbb R)$ is called {\it $d$-Fuchsian} if $\rho = \iota \circ \rho_0$ where $\iota \from \SL(2,\mathbb R) \to \SL(d,\mathbb R)$ is the Veronese embedding and $\rho_0 \in T(S)$. A homomorphism $\rho \from \pi_1(S) \to \SL(d,\mathbb R)$ is a {\it Hitchin homomorphism} if it can be deformed into a $d$-Fuchsian homomorphism. 

Now we define the deformation spaces of Anosov representations on which the pressure metric will be built.
If $G$ is a reductive subgroup of $\SL(d,\mathbb R)$, an element of $G$ is {\it generic} if its centralizer is a maximal torus in $G$. In particular, an element of $\SL(d,\mathbb R)$ is generic if and only if it is diagonalizable over $\mathbb C$ with distinct eigenvalues. We say that a representation $\rho \from \Gamma \to G$ is {\it $G$-generic} if the Zariski closure of $\rho(\Gamma)$ contains a generic element of $G$. 
We denote by ${\rm Hom}(\Gamma,G)$ the space of homomorphisms from $\Gamma$ to $G$.
We say that $\rho \in {\rm Hom}(\Gamma,G)$ is {\it regular} if it is a smooth point of the algebraic variety ${\rm Hom}(\Gamma,G)$.

Denote by $\mathcal{C}(\Gamma,d)$ the space of conjugacy classes of regular, irreducible, projective Anosov representations of $\Gamma$ into $\SL(d,\mathbb R).$ Denote by $\mathcal{C}_g(\Gamma,G)$ the space of conjugacy classes of $G$-generic, regular, irreducible, projective Anosov representations. As shown in \cite[Section 7]{BCLS}, these spaces are real analytic manifolds.

If $\rho$ is a projective Anosov representation, we can associate to each conjugacy class $[\gamma]$ of $\gamma \in \Gamma$ its {\it spectral radius} $\Lambda(\gamma)(\rho)$, i.e., the spectral radius of $\rho(\gamma)$.  For $T \ge 0$, define
$$R_T(\rho) := \left\{ [\gamma] : \log(\Lambda(\gamma)(\rho)) \le T \right\}.$$
We note that the cardinality $\#(R_T(\rho))$ of $R_T(\rho)$ is finite for any $T>0$; see \cite[Proposition 2.8]{BCLS}. We define the {\it entropy} $h(\rho)$ of $\rho$ by
$$h(\rho) := \lim_{T \to \infty} \frac{1}{T} \log \#(R_T(\rho)).$$
The entropy is equal to the Hausdorff dimension of the limit set for convex cocompact representations into Lie groups of rank one; see the discussion after \cite[Corollary 1.7]{BCLS}.

If $\rho_1$ and $\rho_2$ are two projective Anosov representations, we define their {\it intersection number} ${\bf I}(\rho_1, \rho_2)$ by
$${\bf I}(\rho_1,\rho_2) := \lim_{T \to \infty} \left( \frac{1}{\#(R_T(\rho_1))} \sum_{[\gamma] \in R_T(\rho_1)} \frac{\log(\Lambda(\gamma)(\rho_2))}{\log(\Lambda(\gamma)(\rho_1))} \right).$$
The {\it normalized intersection number} ${\bf J}(\rho_1, \rho_2)$ is defined by
$${\bf J}(\rho_1, \rho_2) : = \frac{h(\rho_2)}{h(\rho_1)}{\bf I}(\rho_1,\rho_2).$$

We denote by $\mathrm{Out}(\Gamma)$ the group of outer automorphisms of $\Gamma$.

\begin{thm}\label{thm:AnsovReps}
Fix a word hyperbolic group $\Gamma$.
\begin{enumerate}
    \item The map ${\bf J}$ is $\mathrm{Out}(\Gamma)$-invariant and analytic \cite[Theorem 1.3]{BCLS}, and
    \[
        {\bf J}(\rho_1,-)\from \mathcal{C}(\Gamma,d) \to \Rbb
    \]
    has a unique global minimum at $\rho_1$ \cite[Theorem 1.1]{BCLS}. Hence the Hessian of ${\bf J}(\rho_1,-)$ defines a positive semi-definite {\color{black} symmetric bilinear form $\langle \cdot, \cdot \rangle_{\bf J}$ on the tangent space $T_{\rho_1}\mathcal{C}(\Gamma,d)$.
    \item $\langle \cdot, \cdot \rangle_{\bf J}$ is non-degenerate on $\mathcal{C}_g(\Gamma,d)$, i.e., it is a Riemannian metric \cite[Theorem 1.4]{BCLS}.}
\end{enumerate}
\end{thm}



\cite[Section 9]{BCLS} discusses more details of degenerating vectors of the metrics outside of $\mathcal{C}_g(\Gamma,d)$.

\subsection*{Relating Anosov representations to thermodynamic formalism}
Given an Anosov representation $\rho$, Bridgeman-Canary-Labourie-Sambarino constructed a dynamical system associated to $\rho$ which allows the construction of a {\it thermodynamic mapping} from the deformation space of Anosov representations into the space of cohomology classes of pressure zero real-valued H\"older continuous functions on a symbolic space. Therefore, one can define (the pullback of) the pressure metric in Section \ref{sec:Bridgeman} on the deformation spaces of Anosov representations. It turns out that the Riemannian metric $\langle\cdot,\cdot\rangle_{\bf J}$ is equal to the (the pullback of) the pressure metric $\langle\cdot,\cdot\rangle_{pm}$; see \cite[Proposition 3.11]{BCLS}.
We begin by describing Bridgeman-Canary-Labourie-Sambarino's construction which associates a (uniformly hyperbolic) dynamical system to each representation $\rho$.

Fix a word hyperbolic group $\Gamma$. Denote by $U_0\Gamma$ the Gromov geodesic flow of $\Gamma$; see \cite{Champ94,Gromov87,Min05} for details.  If $\Gamma$ is the fundamental group of a hyperbolic surface $S$, then $U_0\Gamma$ is the geodesic flow on the unit tangent bundle of $S$.
Given a projective Anosov representation $\rho \from \Gamma \to \SL(d,\mathbb R)$, Bridgeman-Canary-Labourie-Sambarino  associated a {\it geodesic flow} $(U_{\rho}\Gamma, \{\phi_t\}_{t\in\mathbb R})$ to $\rho$, which is
H\"older orbit equivalent to the geodesic flow $U_0\Gamma$ of $\Gamma$. In particular, there is a H\"older continuous function $f_\rho \from U_0\Gamma \to \mathbb R$ such that the reparametrization of the Gromov geodesic flow $U_0\Gamma$ by $f_\rho$ is conjugate to the geodesic flow $U_\rho\Gamma$ of $\rho$. See \cite[Section 4]{BCLS} for details.

\begin{prop}[{\cite[Proposition 5.1]{BCLS}}]\label{prop27}
If $\rho \from \Gamma \to \SL(d,\mathbb R)$ is a projective Anosov representation, then the geodesic flow $(U_\rho\Gamma, \{\phi_t\}_{t\in\mathbb R})$ is a topologically transitive metric Anosov flow. 
\end{prop}

We note that metric Anosov flows are a natural generalization of Anosov flows in the setting of compact metric spaces and were studied by Pollicott \cite{Pol87}.
\textcolor{black}{By Proposition \ref{prop27},} the thermodynamic formalism of topologically transitive metric Anosov flows can be applied to define a pressure form on the deformation space of Anosov representations.
Given $\rho \in \mathcal{C}(\Gamma,d)$, the pressure of the H\"older continuous function $(-h(\rho) f_\rho) \from U_0\Gamma \to \mathbb R$ satisfies $$\Pcal(-h(\rho) f_\rho) = 0.$$
Moreover, if $[\rho_1]=[\rho_2]$ then we have $[-h(\rho_1) f_{\rho_1}] = [-h(\rho_2) f_{\rho_2}]$; see \cite[Section 3]{BCLS}. Denote by $\mathcal{C}(U_0\Gamma)$ the set of cohomology classes of H\"older continuous functions on $U_0\Gamma$ with pressure zero.
Then there is a well-defined {\it thermodynamic mapping}
$$\mathcal{E} \from \mathcal{C}(\Gamma, d) \to \mathcal{C}(U_0\Gamma)$$ given by
$$\rho \mapsto [-h(\rho) f_\rho(x)].$$
If $\{\rho_t\}$ is a smooth family of projective Anosov representations and $\{f_t\}$ is an associated smooth family of H\"older reparametrizations, then the pullback of the pressure form equals the Hessian of the normalized intersection ${\bf J}$, namely,
$$\left|\left| \frac{d}{dt}\bigg|_{t=0} \rho_t\right|\right|_\Pcal^2:=\left|\left| \frac{d}{dt}\bigg|_{t=0} -h(\rho_t) f_t\right|\right|_{pm}^2 = \frac{d^2}{dt^2}\bigg|_{t=0} {\bf J}(\rho_0,\rho_t).$$

\subsection{Pressure {\color{black} forms} on hyperbolic components in the moduli space of rational maps}\label{sec:HeNie}
In this subsection, we survey results in \cite{HeNie_MetricHypComp} about pressure  {\color{black}forms} on hyperbolic components in the moduli spaces of rational maps.

A rational map $f\from  \textcolor{black}{\widehat{\mathbb{C}}} \to   \textcolor{black}{\widehat{\mathbb{C}}}$ is a map of the form
$$f(z) = \frac{p(z)}{q(z)},$$ 
where $p(z)$ and $q(z)$ are polynomials in a complex variable $z$. The {\it degree} of $f$ is the maximum of the degrees of $p(z)$ and $q(z)$, provided that $p(z)$ and $q(z)$ do not have a common divisor. The {\it Fatou set} of a rational map $f$ is the largest open subset of $\textcolor{black}{\widehat{\mathbb{C}}}$ on which the sequence $\{f^n\}_{n \ge 1}$ of iterates of the map forms a normal family in the sense of Montel. The {\it Julia set}, denoted by $\Julia(f)$, is the complement of the Fatou set in $\textcolor{black}{\widehat{\mathbb{C}}}$. Both the Fatou and Julia sets are {\it fully invariant} (i.e., both forward and backward invariant) under the rational map $f$.
A rational map is called {\it hyperbolic} if there exists a constant $C>1$ and a smooth conformal metric $||\cdot||_\mu$ on a neighborhood $U\subset \widehat{\mathbb{C}}$ of $\Julia(f)$ such that
$$||f'(z)||_\mu > C > 1$$
for any $z \in \Julia(f)$.
In other words, the dynamics of $f$ restricted to the Julia set $f\from \Julia(f) \to \Julia(f)$ is {\it uniformly hyperbolic}. 

For an integer $d \ge 2$, denote by $\mathrm{Rat}_d$ (resp. $\mathrm{Poly}_d$) the space of degree-$d$ rational maps (resp. polynomials). Denote by $\mathrm{rat}_d:= \mathrm{Rat}_d/\mathrm{Aut}(\textcolor{black}{\widehat{\mathbb{C}}})$ (resp. $\mathrm{poly}_d :=  \mathrm{Poly}_d/\mathrm{Aut}(\mathbb{C})$) the moduli space of degree-$d$ rational maps (resp. polynomials), modulo the action by conjugation of the group of M\"obius transformations (resp.\@ affine automorphisms). Then  $\mathrm{rat}_d$ is a complex $(2d-2)$-dimensional orbifold, and $\mathrm{poly}_d$ is a complex $(d-1)$-dimensional orbifold. A {\it hyperbolic component} of $\mathrm{rat}_d$ (resp. $\mathrm{poly}_d$) is a connected component of the set of hyperbolic maps.

\subsubsection{Symmetric bilinear forms $\langle\cdot,\cdot\rangle_G$ and semi-norms $||\cdot||_G$ on hyperbolic components in ${\rm Rat}_d$}
Fix a hyperbolic component $\widetilde{\Hcal}$ in ${\rm Rat}_d$ and $f \in \widetilde{\Hcal}$. There exists a neighborhood $U(f)$ of $f$ in $\widetilde{\Hcal}$ such that a quasi-conformal conjugacy $\phi_g \from \mathcal{J}(f) \to \mathcal{J}(g)$ is well-defined. We define $\delta \from \widetilde{\Hcal} \to \mathbb{R}$ as the function sending $g$ to the Hausdorff dimension of $\Julia(g)$.

Denote by $\nu$ the equilibrium measure of the H\"older potential $-\delta(f)\log |f'| \from \mathcal{J}(f) \to \mathbb{R}$, which has zero pressure. That is, $\nu$ is the unique $f$-invariant probability measure on $\Julia(f)$ such that the measure-theoretic entropy $h_\nu(f)$ of $\nu$ equals $\delta(f)\int_{\Julia(f)}\log|f'|d\nu$.
Define a function $\mathrm{Ly}(\nu,\cdot) \from U(f)\to \mathbb R$ by
$$\mathrm{Ly}(\nu,g) :=  \int_{\Julia(g)} \log|g'| d\left((\phi_g)_* \nu\right)=\int_{\Julia(f)} \log|g' \circ \phi_g|d\nu .$$
The function $\mathrm{Ly}(\nu,\cdot) \from U(f)\to \mathbb R$ is harmonic. In particular, it is real-analytic; see \cite[Proposition 2.10]{HeNie_MetricHypComp}.


Now consider the real analytic function $G_{f}\colon U(f) \to \mathbb{R}$ given by 
$$G_f(g)=\delta(g)\mathrm{Ly}(\nu,g).$$
A key property of $G_f$ is that it has local minimum at $f$; namely, for any $g\in U(f)$, we have $$G_f(f)\le G_f(g).$$
See \cite[Proposition 4.1]{HeNie_MetricHypComp}.

Since $G_f$ has local minimum at $f$, the Hessian of $G_f$ at $f$ defines a positive semi-definite symmetric bilinear form $\langle \cdot, \cdot \rangle_G$ on the tangent space $T_f\widetilde{\mathcal{H}}$; see \cite[Section 7]{BridgemanTaylor}. More specifically, choose a smooth $2$-parameter family $\gamma(t,s), t,s \in (-\epsilon, \epsilon)$ in $U(f)$ with $\gamma(0) = f$ and $\frac{\partial}{\partial t} \big|_{t=0}\gamma(t,0)=w \in T_f\widetilde{\mathcal{H}}$ and $\frac{\partial}{\partial s} \big|_{s=0}\gamma(0,s)=v \in T_f\widetilde{\mathcal{H}}$ .
Define
$$\langle w,v \rangle_G := \partial_{wv}^2 G_{f}= \frac{\partial}{\partial t} \Big|_{t=0}\frac{\partial}{\partial s} \Big|_{s=0}G_{f}(\gamma(t,s)).$$ 
Define a semi-norm $||\cdot||_G$ by $||v||_G:=\sqrt{\langle v,v\rangle_G}$ for $v\in T_f \widetilde{\mathcal{H}}$.

\subsubsection{Semi-norms $||\cdot||_G$ on hyperbolic components in ${\rm rat}_d$}
Consider a hyperbolic component $\mathcal{H}$ in the moduli space $\mathrm{rat}_d$. Denote by $\widetilde{\mathcal{H}}$ the corresponding hyperbolic component in ${\rm Rat}_d$.
Suppose that $(f_t)_{t\in(-\epsilon,\epsilon)}$ and $(g_t)_{t\in(-\epsilon,\epsilon)}$ are smooth paths in ${\rm Rat}_d$ so that $([f_t])$ and $([g_t])$ yield the same path in $\mathcal{H}$. Then $g_t=\gamma_t \cdot f_t \cdot \gamma_t^{-1}$ for a smooth family $\{\gamma_t\}$ of M\"obius transformations. Then the H\"{o}lder potential functions $-\delta(f_t)\log|f'_t|$ and $-\delta(g_t)\log|g'_t|$ are cohomologous so that $G_{g_t}=G_{f_t}$ for any $t$. The family $(h_t)$ also conjugates the neighborhoods $U(f_t)$ and $U(g_t)$ of $f_t$ and $g_t$. Hence the semi-norm $||\cdot||_G$ on $\widetilde{\mathcal{H}}$ descends to a semi-norm on $\mathcal{H}$. Abusing notation, we also denote the semi-norm on $\mathcal{H}$ by $||\cdot||_G$. The $2$-form  $\langle\cdot,\cdot\rangle_G$ requires a more careful argument when $[f_0]$ is an orbifold point of ${\rm rat}_d$, which we do not address in this article.

\subsubsection{Conformal equivalence}
Since $f \from \mathcal{J}(f) \to \mathcal{J}(f)$ is uniformly hyperbolic for any $[f] \in \mathcal{H}$ and $(\mathcal{J}(f_1),f_1)$ and $(\mathcal{J}(f_2),f_2)$ are topologically conjugate for any two points $[f_1],[f_2] \in \mathcal{H}$,
there exists a one-sided subshift of finite type (or symbolic coding) $\Sigma_A^+$ such that for every $[f] \in \mathcal{H}$, $(\mathcal{J}(f),f)$ is conjugate to $(\Sigma_A^+,\sigma)$. We denote by $\pi_f\from\Sigma_A^+ \to \Julia(f)$ the semi-conjugacy of symbolic coding.

Define $\Ecal\from \mathcal{H} \to \mathcal{C}(\Sigma_A^+)$ by
    \[
        [f] \mapsto \left[- \delta(f) \log |f' \circ \pi_f|\right].
    \]
For $[f]\in\mathcal{H}$ and $v\in T_{[f]}\mathcal{H}$, choose a smooth real $1$-dimensional curve  $c(t)$ in $\mathcal{H}$ defined on $(-\epsilon, \epsilon)$ with $c(0)=[f]$ and $c'(0)=v$. Recall that the pullback of the pressure {\color{black} norm} is given by
$$||v||_{\mathcal{P}} := \left|\left|\frac{d}{dt}\bigg|_{t=0}\mathcal{E}(c(t))\right|\right|_{\mathcal{P}}.$$
We call $||\cdot||_\Pcal$ the {\it pressure semi-norm} on $T_{[f]}\Hcal$.

It is proven in {\cite[Proposition 4.2]{HeNie_MetricHypComp} that  $||\cdot||_G$ and $||\cdot||_{\mathcal{P}}$ are conformally equivalent. More precisely, fixing the notations as above, we have
	$$||v||_{\mathcal{P}}^2 = \frac{||v||_G^2}{\delta(f)\int_{\Sigma_A^+}\log|f'\circ \pi_f| d\nu(x)}$$
 where $\nu$ is the equilibrium measure of $-\delta(f) \log|f'\circ \pi_f|$.

\subsubsection{Non-degeneracy conditions for the pressure semi-norm}
{\color{black} A point $x \in \widehat{\mathbb C}$ is a {\it periodic point} of a rational map $f$ of period $n$ if $f^n(x) = x$ and $n$ is the smallest integer satisfying this condition. The {\it multiplier} of a periodic point $x$ (or of the periodic orbit $\{x,\ldots,f^{n-1}(x)\}$) is $(f^n)'(x)$. A multiplier is {\it repelling} if $|(f^n)'(x)|>1$.
}

\begin{thm}[{\cite[Theorem 1.1]{HeNie_MetricHypComp}}] \label{thm_1.1}
Suppose thet $\Hcal$ is a hyperbolic component in ${\rm rat}_d$. For $[f]\in \Hcal$, if $[f]$ has a repelling multiplier that is not a real number, then $||\cdot||_G$ and $||\cdot||_\mathcal{P}$ are non-degenerate at $[f]$, i.e., they are norms on $T_{[f]} \mathcal{H}$. In particular, if $\delta(f)\in(1,2)$, then $||\cdot||_G$ and $||\cdot||_\mathcal{P}$ are non-degenerate at $[f]$.  
\end{thm}

Theorem \ref{thm_1.1} is proved by using the following proposition, which provides a constraint on repelling multipliers when the semi-norms are degenerate.
\begin{prop}[{\cite[Corollary 4.3]{HeNie_MetricHypComp}}] \label{Prop:DegenerateVector}
Fix a hyperbolic component $\widetilde{\Hcal}$ in ${\rm Rat}_d$ and fix $f \in \widetilde{\Hcal}$. Then for any $v \in T_f(\widetilde{\Hcal})$, the following are equivalent:
    \begin{enumerate}
        \item $||v||_\Pcal=0$.
        \item $\left.\frac{d}{dt}\right|_{t=0} \delta(f_t) \log{|\lambda_C(f_t)|}=0$ for the multiplier $\lambda_C(f_t)$ of any repelling cycle $C$. Here $\{f_t\}_{t \in (-1,1)}$ is any smooth curve in $\Hcal_f$ such that $f_0 = f$ and $\frac{d}{dt}\big|_{t=0} f_t = v$.
    \end{enumerate}
\end{prop}

Theorem \ref{thm_1.1} has the following corollary for hyperbolic components in $\mathrm{poly}_d$. 
\begin{coro}[{\cite[Corollary 1.3]{HeNie_MetricHypComp}}]\label{cor_SL}
Fix a hyperbolic component  $\mathcal{H}$ in $\mathrm{poly}_d$ that is neither the central hyperbolic component nor the shift locus. Then $||\cdot||_G$ is non-degenerate.
\end{coro}
{\color{black} 
Recall that in $\mathrm{poly}_d$, the {\em central hyperbolic component} is the hyperbolic component containing $z \mapsto z^d$, and the {\em shift locus} is the hyperbolic component consisting of polynomials of which all the critical points are in the attracting basin of the infinity.}

We will see in the next section that the condition in Theorem \ref{thm_1.1} is sharp. In particular, we consider the space $\mathcal{QB}_d^{fm}$ of conjugacy classes of degree-$d$ quasi-Blaschke products with marked fixed points, which will be seen as a hyperbolic component.
This hyperbolic component does not satisfy the condition in Theorem \ref{thm_1.1} and we show that the pressure semi-norm is not positive definite.


\section{Degeneracy loci of the pressure semi-norms in quasi-Blaschke products spaces}\label{sec_DegeLoci_QB}
In this section, we study the degeneracy loci of the pressure  semi-norm in the space $\QB_d^{fm}$ of conjugacy classes of degree $d \ge 2$ quasi-Blaschke products with marked fixed points. In Section \ref{sec_defQBd}, we define the space $\QB_d^{fm}$; see Definition \ref{def_qbd}. In Section \ref{sec_3.2}, we prove the main result on the degeneracy locus of the pressure semi-norm $||\cdot||_\Pcal$; see Theorem \ref{thm_main}.

\subsection{The space of fixedpoint-marked quasi-Blaschke products}\label{sec_defQBd}
In this section, we investigate the space $\QB_d^{fm}$ of conjugacy classes of degree $d \ge 2$ quasi-Blaschke products with marked fixed points.

\subsection*{Fixedpoint-marked rational maps} For $d\ge 2$, any degree-$d$ rational map $f$ has $d+1$ fixed points $x_1,x_2,\dots,x_{d+1}$ counted with multiplicity. A rational map $f$ together with an ordered $(d+1)$-tuple of its fixed points $(f; x_1,x_2,\dots,x_{d+1})$ is called a {\it rational map with marked fixed points} or a {\it fixedpoint-marked rational map}. For simplicity, we sometimes omit $x_i$'s and say that $f$ is a fixedpoint-marked rational map when the marking of the fixed points are understood or inessential in the context. Denote by ${\rm Fix}(f)$  the set of $(d+1)$ fixed points with multiplicity.

{\color{black} We define the space of degree-$d$ rational maps with marked fixed points ${\rm Rat}_d^{fm}$ by
\[
    {\rm Rat}_d^{fm}:= \left\{(f; x_1,x_2,\dots,x_{d+1})\in {\rm Rat}_d \times \widehat{\mathbb{C}}^{d+1}~|~\{x_1,x_2,\dots,x_{d+1}\}={\rm Fix}(f)\right\}.
\] By \cite[Lemma 9.2]{Milnor_HypComp}, $\Rat_d^{fm}$ is a \textcolor{black}{ complex  manifold.} 
We remark that having $d+1$ distinct fixed points is equivalent to having no fixed points with multiplier $1$.
}

{\color{black} A fixedpoint-marked rational map $(f;x_1,x_2,\dots,x_{d+1})$ is {\it hyperbolic} if $f$ is a hyperbolic rational map. A hyperbolic component in ${\rm Rat}_d^{fm}$ is a connected component of the subset of hyperbolic fixedpoint-marked rational maps.}

M\"obius transformations {\color{black} in} $\PSL(2,\mathbb{C})$ act on ${\rm Rat}^{fm}_d$ by
\[
    \phi\cdot (f;x_1,x_2,\dots,x_{d+1}):= (\phi \circ f \circ \phi^{-1}; \phi(x_1),\phi(x_2),\dots,\phi(x_{d+1}))
\]
so that the action is free on the complement of the locus consisting of rational maps having less than three fixed points. In particular, the action is free on the set of hyperbolic fixedpoint-marked rational maps. We denote by ${\rm rat}_d^{fm}$ the quotient of ${\rm Rat}_d^{fm}$ by the $\PSL(2,\mathbb{C})$-action and call it the space of conjugacy classes of degree-$d$ rational maps with marked fixed points. We refer the reader to \cite[Section 9]{Milnor_HypComp} for details on fixedpoint-marked rational maps.

\begin{defn}[Quasi-Blaschke products]
A hyperbolic rational map $f$ is said to be a {\it quasi-Blaschke product} if its Julia set $\mathcal{J}(f)$ is a quasi-circle, and $f$ fixes each of the two Fatou components.  
\end{defn}

We define an \textcolor{black}{open} subset $\widetilde{\mathcal{U}\QB}^{fm}_d$ of ${\rm Rat}_d^{fm}$ by
$$\widetilde{\mathcal{U}\QB}^{fm}_d := \left\{ (f;x_1,x_2,\dots,x_{d+1}) \in {\rm Rat}_d^{fm} : f \text{ is a quasi-Blaschke product}\right\}.$$
{\color{black} Consider the set of quasi-Blaschke products $\widetilde{\QB}_d$ in ${\rm Rat}_d$. One can show that $\widetilde{\QB}_d$ is the hyperbolic component containing $z \mapsto z^d$, for example, by using \cite[Proposition 5.5]{McM_AutRat} and the connectivity of the space of Blaschke products. Similarly, one can show that $\widetilde{\mathcal{UQB}}_d^{fm}$ is the union of the hyperbolic components containing $z \mapsto z^d$ with marked fixed points.}

Let us discuss the connected components of $\widetilde{\mathcal{U}\QB}_d^{fm}$. For two ordered sets $A=\{\alpha_1,\alpha_2,\dots,\alpha_n\}$ and $B=\{\beta_1,\beta_2,\dots,\beta_n\}$, we say that {\em $A$ and $B$ have the same cyclic order (resp.\@ reversed cyclic orders)} if there exists an integer $k$ such that $\alpha_{i}=\beta_{i+k}$ (resp.\@ $\alpha_i=\beta_{-i+k}$) for any $i\in \{1,2,\dots,n\}$ where the indices are considered modulo $n$. Denote by $I$ the set obtained as the quotient of the $(d+1)$-symmetric group $S_{d+1}$ by an equivalence relation $\sim$ defined as follows: for $\sigma,\sigma' \in S_{d+1}$,
$\sigma\sim \sigma'$ if and only if
\begin{enumerate}
    \item $\sigma(1)=\sigma'(1)$, $\sigma(2)=\sigma'(2)$, and the ordered sets $\{\sigma(3),\sigma(4),\dots, \sigma(d+1)\}$ and\\ $\{\sigma'(3),\sigma'(4),\dots, \sigma'(d+1)\}$ have the same cyclic order, or
    \item $\sigma(1)=\sigma'(2)$, $\sigma(2)=\sigma'(1)$, and the ordered sets $\{\sigma(3),\sigma(4),\dots, \sigma(d+1)\}$ and\\ $\{\sigma'(3),\sigma'(4),\dots, \sigma'(d+1)\}$ have the reversed cyclic orders.
\end{enumerate}

{\color{black} We define a map $\Phi\from \widetilde{\mathcal{U}\QB}_d^{fm} \to I$ as follows. Take $(f; x_1,x_2,\dots,x_{d+1})\in \widetilde{\mathcal{U}\QB}_d^{fm}$. Then there exists a M\"{o}bius transformation $\phi$ so that $f^\phi:=\phi \circ f \circ \phi^{-1}$ has two attracting fixed points at $0$ and $\infty$. The Julia set $\Julia(f^\phi)$ is a quasi-circle on the plane. We define $\Phi((f; x_1,x_2,\dots,x_{d+1}))$ to be an element $\sigma$ in $S_{d+1}$ such that $\phi(x_{\sigma(1)})=0$, $\phi(x_{\sigma(2)})=\infty$ and the set of points $\{\phi(x_{\sigma(3)}),\phi(x_{\sigma(4)}),\dots, \phi(x_{\sigma(d+1)})\}$ is ordered counter-clockwise along the quasi-circle $\Julia(f^\phi) \subset \mathbb{C}$.
We note that the element $\sigma\in S_{d+1}$ is well-defined up to the equivalence class defined above; exchanging two attracting fixed points reverses the cyclic order.}

Moreover, the map $\Phi\from \widetilde{\mathcal{U}\QB}_d^{fm} \to I$ is continuous because of the holomorphic motion of the fixed points \cite{MSS_DynRatMap,Lyu83typical}. Hence we have a function $\Phi_*\from \pi_0\left( \widetilde{\mathcal{U}\QB}_d^{fm}\right) \to I$.

\begin{lem} \label{lem_bijective}
The function
    $\Phi_*\from \pi_0\left( \widetilde{\mathcal{U}\QB}_d^{fm}\right) \to I$ is bijective.
\end{lem}
\begin{proof}
The surjectivity follows from the fact that ${\rm Rat}_d^{fm}$ contains all the combinations of markings of fixed points. {\color{black} More precisely, for a fixed $\sigma \in I$, there exists a quasi-Blaschke product $f$ with marked fixed points $x_1,\ldots,x_{d+1}$ such that $x_{\sigma(1)} =0,x_{\sigma(2)} =\infty$ and the set $\{x_{\sigma(3)},\ldots,x_{\sigma(d+1)} \}$ is ordered counter-clockwise on the quasi-circle. Such an $(f;x_1,\ldots,x_{d_1})$ exists as $\widetilde{\mathcal{U}\QB}_d^{fm}$ contains, for the same map $f$, all the possible ways of marking its fixed points.}

We claim that $\Phi_*$ is injective. Suppose that two maps $(f;x_1(f),x_2(f),\dots,x_{d+1}(f))$ and $(g;x_1(g),x_2(g),\dots,x_{d+1}(g))$ in $\widetilde{\mathcal{U}\QB}^{fm}_d$ with $\Phi(f)=\Phi(g)=[\sigma] \in S_{d-1}/\sim$ are normalized in such a way that they have fixed points at $0=x_{\sigma(1)}(f)=x_{\sigma(1)}(g)$ and $\infty=x_{\sigma(2)}(f)=x_{\sigma(2)}(g)$. By \cite[Corollary 3.6]{McM_AutRat} and \cite[Theorem 2.9]{McMSul_QCIII:TeichHoloDyn}, through continuous deformations fixing $0$ and $\infty$, we can deform $f$ and $g$ to be post-critically finite, i.e., $0$ and $\infty$ are degree-$d$ super-attracting fixed points. Continuous deformations preserve the cyclic ordering of repelling fixed points on the Julia sets because the collision of fixed points yields a parabolic fixed point. Then {\color{black} $f(z) = g(z) = z^d$ as $z^d$ is the only post-critically finite map having $0$ and $\infty$ as degree-$d$ super-attracting fixed points. Moreover, the markings of the fixed points of $f$ and $g$ coincide. Therefore, $\Phi_*$ is injective. This completes the proof.} 
\end{proof}

Define $\widetilde{\QB}^{fm}_{d,[\sigma]}:= \Phi^{-1}([\sigma])$. Then $\widetilde{\QB}^{fm}_{d,[\sigma]}$ is a connected component of $\widetilde{\mathcal{U}\QB}^{fm}_d$. 
\begin{prop}
We have 
    \[
        \widetilde{\mathcal{UQB}}^{fm}_d = \bigcup_{[\sigma]\in I} \widetilde{\mathcal{QB}}^{fm}_{d,[\sigma]}.
    \]
    In particular, $\widetilde{\mathcal{UQB}}_d$ has $\frac{(d+1)!}{2(d-1)}$ connected components.
\end{prop} 
\begin{proof}
The proposition follows from Lemma \ref{lem_bijective} and the fact that $|I|=\frac{(d+1)!}{2(d-1)}$.
\end{proof}

{\color{black} Since $\widetilde{\mathcal{QB}}^{fm}_{d,[\sigma]}$ is a connected component of $\widetilde{\mathcal{U}\QB}^{fm}_d$, it is a hyperbolic component in ${\rm Rat}_d^{fm}$.}

\begin{defn}\label{def_qbd}
For any $[\sigma]\in S_{d+1}/\sim$, we define the hyperbolic component of quasi-Blaschke products associated to $[\sigma]$ in the moduli space of fixedpoint-marked rational maps by
$$\QB^{fm}_{d,[\sigma]} := \widetilde{\QB}^{fm}_{d,[\sigma]}/\PSL(2,\mathbb{C}) \subset {\rm rat}_d^{fm}.$$
By convention, we define $\widetilde{\QB}^{fm}_d:=\widetilde{\QB}^{fm}_{d,[id]}$ and $\QB^{fm}_d:=\QB^{fm}_{d,[id]}$ where $id$ is the identity element of $S_{d+1}$.
\end{defn}

\begin{lem}
$\widetilde{\QB}^{fm}_d$ and $\QB^{fm}_d$ are complex manifolds, and they are hyperbolic components of ${\rm Rat}_d^{fm}$ and ${\rm rat}_d^{fm}$, respectively.
\end{lem}
\begin{proof}
We claim that $\widetilde{\mathcal{QB}_d}^{fm}$ is a complex manifold. \textcolor{black}{Recall that $\widetilde{\mathcal{UQB}_d}^{fm}$ is an open subset of $\Rat_{d}^{fm}$ which is a complex manifold. Hence, } $\widetilde{\mathcal{UQB}_d}^{fm}$ and $\widetilde{\mathcal{QB}_d}^{fm}$ are also complex manifolds.
We claim that $\QB_d^{fm}$ is a complex manifold without any orbifold singular points. Indeed, since any quasi-Blaschke product has at least three fixed points, the $\PSL(2,\mathbb{C})$-action on $\widetilde{\mathcal{U}\QB}_d^{fm}$ is free. Therefore $\QB_d^{fm}$ is a complex manifold.
Moreover, $\QB_d^{fm}$ is a hyperbolic component in ${\rm rat}_d^{fm}$ as $\widetilde{\QB}_d^{fm}$ is a hyperbolic component in ${\rm Rat}_d^{fm}$.
\end{proof}



\medskip

For any $[(f,x_1,x_2,\dots,x_{d+1})]\in \QB_d^{fm}$ there exists a unique representative $f$ in the conjugacy class such that $x_1=0,x_2=\infty$, and $x_3=1$. Then $f$ is of the form
\[
    f(z)=Q_{\bf{a},\bf{b}}(z):=\left(\prod_{j=1}^{d-1} \frac{1+b_j}{1+a_j}\right) z \prod_{j=1}^{d-1}\frac{z+a_j}{1+b_j z},
\]
where ${\bf a}:=(a_1,a_2,\dots,a_{d-1}) \in \mathbb{C}^{d-1}$ and ${\bf b}:=(b_1,b_2,\dots,b_{d-1}) \in \mathbb{C}^{d-1}$. The marking of the fixed points of $Q_{\bf{a},\bf{b}}$ is given by $x_0=0$, $x_1=\infty$ and $\{x_2=1,x_3,\dots,x_d\}$, where the set $\{x_2=1,x_3,\dots,x_d\}$ is counter-clockwisely ordered on the quasi-circle Julia set of $Q_{\bf{a},\bf{b}}$.
This gives rise to a holomorphic embedding $\Psi\from \QB_d^{fm} \to \mathbb{C}^{d-1}/S_{d-1} \times \mathbb{C}^{d-1}/S_{d-1}$ defined by $\Psi(Q_{\bf{a},\bf{b}})=([\bf{a}],[\bf{b}])$. More precisely,
\[
    \Psi\from \left[\left(\prod_{j=1}^{d-1} \frac{1+b_j}{1+a_j}\right) z \prod_{j=1}^{d-1}\frac{z+a_j}{1+b_j z}\right] \mapsto \left([(a_1,a_2,\dots,a_{d-1})],[(b_1,b_2,\dots,b_{d-1})]\right).
\]
We remark that $\Psi$ is not surjective, i.e., $Q_{\bf{a},\bf{b}}$ may not be a quasi-Blaschke product for an arbitrary pair ${\bf a},{\bf b} \in \mathbb{C}^{d-1}$. For any vector ${\bf v}=(v_1,v_2,\dots, v_{d-1})\in \mathbb{C}^{d-1}$, define  $\bar{\bf v}$ to be the entry-wise complex conjugate of ${\bf v}$, i.e., $\overline{\bf v}=(\overline{v_1},\overline{v_2},\dots,\overline{v_{d-1}})$.

\begin{lem}\label{lem:involution}
The map $Q_{{\bf a},{\bf b}}(z)$ is a quasi-Blaschke product if and only if $Q_{\overline{\bf b},\overline{\bf a}}(z)$ is a quasi-Blaschke product. Moreover, $Q_{{\bf a},{\bf b}}(z)$ is in the hyperbolic component $\mathcal{QB}_d^{fm}$ if and only if $Q_{\overline{\bf b},\overline{\bf a}}(z)$ is in the same hyperbolic component $\mathcal{QB}_d^{fm}$.
\end{lem}
\begin{proof}
    The lemma follows from the facts that $Q_{{\bf a},{\bf b}}(z)$ and $Q_{\overline{\bf b},\overline{\bf a}}(z)$ are conjugate by the anti-holomorphic involution $z \mapsto \frac{1}{\bar{z}}$, and that the conjugation fixes $0,1,\infty$ and preserves the cyclic ordering of repelling fixed points on the Julia sets.
\end{proof}

It follows from Lemma \ref{lem:involution} that the image $\Psi(\QB_d^{fm})$ is invariant under the anti-holomorphic involution $\iota\from \mathbb{C}^{d-1}/S_{d-1} \times \mathbb{C}^{d-1}/S_{d-1} \to \mathbb{C}^{d-1}/S_{d-1} \times \mathbb{C}^{d-1}/S_{d-1}$ defined by
\[
    \iota([{\bf a}],[{\bf b}])= ([\bar{{\bf b}}],[\bar{{\bf a}}]).
\]

Define $\Bcal_d^{fm} \subset \QB_d^{fm}$ by
\[
    \Bcal_d^{fm} : =\left\{[f] \in \QB_d^{fm} ~|~ f \text{ is a Blaschke product} \right\}.
\]
Then $\Psi(\Bcal_d^{fm})$ is the locus of fixed points of $\iota$.
We have the following decomposition of the tangent spaces of $\QB_d^{fm}$ at points in $\Bcal_d^{fm}$.

\begin{lem}\label{lem_3.9}
Fix $[f] \in \Bcal_d^{fm}$. Then we have
$$T_{[f]}\mathcal{\QB}_d^{fm} = T_{[f]} \Bcal_d^{fm} \oplus J\cdot T_{[f]} \Bcal_d^{fm}$$
where $J$ is the (almost) complex structure on ${\rm rat}_d^{fm}$.
\end{lem}
\begin{proof}
    Denote by $(\mathbb{R}^{2n},\tilde{J})$ a real $2n$-dimensional vector space with an almost complex structure $\tilde{J}$. Suppose that $\phi\in {\rm GL}(2n,\mathbb{R})$ is an anti-holomorphic involution of $(\mathbb{R}^{2n},\tilde{J})$, i.e., $\phi\circ \phi=id$ and $\phi(\tilde{J}\cdot v)=-\tilde{J}\cdot \phi(v)$. Since involutions are diagonalizable, $\mathbb{R}^{2n}$ is decomposed into $V\oplus \tilde{J}\cdot V$ where $V$ and $\tilde{J}\cdot V$ are eigen-spaces of $\iota$ having eigen-values $1$ and $-1$, respectively.

Denote by $J_1$ the almost complex structure on $\mathbb{C}^{d-1}/S_{d-1} \times \mathbb{C}^{d-1}/S_{d-1}$. Since $\Psi$ is biholomorphic onto its image, we can identify $(\QB_d^{fm},J)$ with its image $(\Psi(\QB_d^{fm}),J_1)$.

Since {\color{black} the differential $D\iota$} defines an anti-holomorphic involution of the tangent space \\
$(T_{\Psi([f])}\Psi(\mathcal{\QB}_d^{fm}),J_1) = (T_{[f]}\mathcal{\QB}_d^{fm}, J)$
such that the subspace $T_{[f]} \Bcal_d^{fm}$ coincides with the locus of fixed points of $D\iota$, the lemma follows from the first paragraph.
\end{proof}

\subsection*{Analogues of pure shearing and pure bending tangent vectors}
The decomposition of the tangent space of quasi-Blaschke products in Lemma \ref{lem_3.9} is analogous to that of quasi-Fuchsian groups in Equation \eqref{eqn:DecompTQF}. However, geometric interpretations of tangent vectors in $T_{[f]}\Bcal_d^{fm}$ and $J \cdot T_{[f]}\Bcal_d^{fm}$ are unknown.

\subsection{Degeneracy locus of $||\cdot||_\Pcal$ in $\QB_d^{fm}$}\label{sec_3.2}
We continue to use notations defined in Section \ref{sec_defQBd}. The construction in Section \ref{sec:Survey} of the pressure  semi-norms also defines a pressure  semi-norms on a hyperbolic component of ${\rm Rat}_d^{fm}$.
Since $\QB_d^{fm}$ is a hyperbolic component, the pressure  {\color{black} semi-norm} $||\cdot||_\Pcal$ is defined on tangent spaces of $\QB_d^{fm}$. In this section, we study the degeneracy locus of $||\cdot||_\Pcal$ on $\QB_d^{fm}$.

 
\begin{thm}\label{thm_main}
Fix $[f]\in \QB_d^{fm}$ and $0\neq v\in T_{[f]} \QB_d^{fm} $. The following properties hold.
\begin{enumerate}
    \item If $[f] \notin  \Bcal_d^{fm}$, then $||v||_\Pcal \neq 0$.
    \item If $[f]\in \Bcal_d^{fm}$ and $v \in J \cdot T_{[f]} \Bcal_d^{fm}$, then $||v||_\Pcal=0$.
    \item Suppose $[f]\in \Bcal_d^{fm}$ and $||\cdot ||_\Pcal$ is non-degenerate in $T_{[f]}\Bcal_d^{fm}$. If $||v||_\Pcal=0$ for some $v\in T_{[f]}\QB_d^{fm}$, then $v\in J\cdot T_{[f]}\Bcal_d^{fm}$.
\end{enumerate}
\end{thm}
\begin{proof}
(1) follows from Theorem \ref{thm_1.1} because for $[f] \notin  \Bcal_d^{fm}$, the Hausdorff dimension of the Julia set $\delta(f)$ is strictly bigger than $1$.

To prove (2) and (3), we use similar arguments as in \cite[Section 7]{Bridgeman_WPMetricQF}. Let us first prove (2). Fix $[f] \in \Bcal_d^{fm}$. Recall that there is a neighborhood $U(f)$ of $f$ on which the conjugacy $\phi_{f,g} \from \Julia(f) \to \Julia(g)$ is well-defined. For a repelling cycle $C$ of $f$, we define
$L_C\from U(f) \to \Rbb$ by
$$L_C([g]):=\log{|\lambda_C(g)|}$$
where $\lambda_C(g)$ is the multiplier of the cycle $\phi_{f,g}(C)$ in $\Julia(g)$. Similarly, we define $\Lcal_C\from U(f) \to \Cbb$ by
\[
\Lcal_C([g]):=\log{\lambda_C(g)}
\]
so that $L_C=\Re(\Lcal_C)$ where $\Re$ denotes the real part of a complex number.

Fix $v\in J\cdot T_{[f]}\Bcal_d^{fm}$, i.e., $v=J\cdot w$ for some $w\in T_{[f]}\Bcal_d^{fm}$. Denote by $D L_C(v)$ and $D \Lcal_C(v)$ the derivatives of $L_C$ and $\Lcal_C$ along the tangent vector $v$, respectively. Then, for any repelling cycle $C$, we have
\begin{align*}
    D L_C(v)=\Re(D \Lcal_C(v))&=\Re( D\Lcal_C(J\cdot w)) &\\
    &= \Re(i\cdot D \Lcal_C(w))& ({\rm\because \Lcal_C~is~holomorphic})\\
    &=\Re(i \cdot D L_C(w))& (\because D L_C(w)=D \Lcal_C(w)\in\Rbb)\\
    &=0.&
\end{align*}
Choose a smooth path $\{[f_t]\}_{t \in (-\epsilon,\epsilon)}$ in $\Bcal_d^{fm}$ representing $v$; namely, $\{[f_t]\}$ is such that
$[f_0]=[f]$ and $v=\frac{d}{dt}|_{t=0}[f_t]$.  Since $\delta(f_t)=1$ for any $t\in(-\epsilon,\epsilon)$, we have
\[
    \left.\frac{d}{dt}\right|_{t=0} -\delta(f_t) \log|f_t'| = \left.\frac{d}{dt}\right|_{t=0}-\log|f_t'|=0.
\]
Then by Proposition \ref{Prop:DegenerateVector} we obtain  $||v||_{\mathcal{P}} =0.$

Finally, we prove Statement (3). Fix $[f] \in \QB_d^{fm}$ and $v \in T_{[f]}\QB_d^{fm}$ with $||v||_\Pcal=0$. Then by Statement (1), $[f] \in \Bcal_d^{fm}$. By Lemma \ref{lem_3.9}, the tangent space at $[f]$ is decomposed as $T_{[f]}\mathcal{\QB}_d^{fm} = T_{[f]}\Bcal_d^{fm} \oplus J\cdot T_{[f]}\Bcal_d^{fm}$. 
There exist two vectors $v_1,v_2\in T_{[f]}\Bcal_d^{fm}$  such that $v=v_1+ J\cdot v_2$. Then for any repelling cycle $C$ of $f$, we have
\begin{align*}
    D L_C(v) &= \Re(D \mathcal{L}_C(v))\\
    &= \Re(D \mathcal{L}_C(v_1)+ D \mathcal{L}_C(J\cdot v_2))\\
    &=\Re(D \mathcal{L}_C(v_1)+i\cdot D\mathcal{L}_C(v_2)).
\end{align*}
Notice that if $w \in T_{[f]}\Bcal_d^{fm}$, then $D L_C(w) = D \mathcal{L}_C(w)$ is a real number. Therefore, continuing the above calculation, we have
\begin{align*}
\Re(D \mathcal{L}_C(v_1)+i\cdot D \mathcal{L}_C(v_2)) = \Re(D L_C(v_1)+i\cdot D L_C(v_2)) = D L_C(v_1).
\end{align*}
Hence $D L_C(v) = D L_C(v_1)$ for any repelling cycle $C$ of $f$.

On the other hand, since $||v||_{\mathcal{P}}=0$, if $\{[f_t]\}_{t \in (-\epsilon,\epsilon)}$ in $\QB_d^{fm}$ is such that $[f_0]=[f]$ and $\frac{d}{dt}|_{t=0}[f_t] = v$, we have
\begin{align*}
    0 & = \left.\frac{d}{dt}\right|_{t=0} \delta(f_t) \log{|\lambda_C(f_t)|}\\
    & = \left.\frac{d}{dt}\right|_{t=0} \log{|\lambda_C(f_t)|}.
\end{align*}
The second equality follows from the fact that $\delta(f_0) = 1\le \delta(f_t)$ and $\frac{d}{dt}\big|_{t=0}\delta(f_t) = 0$ as \textcolor{black}{$[f] \in \Bcal_{d}^{fm}$}. Hence, we have $L_C'(v) =0$ for any repelling cycle $C$ of $f$. 

Since $D L_C(v) = D L_C(v_1)$, we have $D L_C(v_1) = 0$ for any repelling cycle $C$ of $f$. This implies that $||v_1||_{\mathcal{P}}=0$. Since $||\cdot||_\Pcal$ is non-degenerate on $T_{[f]}\Bcal_d^{fm}$, we obtain that $v_1=0$. Therefore $v = J \cdot v_2$. This completes the proof.
\end{proof}

\bibliography{PressureMetric}
\bibliographystyle{alpha}

\end{document}